\documentclass[reqno, 12pt]{amsart}

\makeatletter
\let\origsection=\section \def\section{\@ifstar{\origsection*}{\mysection}} 
\def\mysection{\@startsection{section}{1}\z@{.7\linespacing\@plus\linespacing}{.5\linespacing}{\normalfont\scshape\centering\S}}
\makeatother

\usepackage{graphicx}
\usepackage{geometry}
\numberwithin{equation}{section}
\numberwithin{figure}{section}

\usepackage{xcolor}
\usepackage{hyperref}
\hypersetup{
    unicode,
    colorlinks,
    linkcolor={red!60!black},
    citecolor={green!60!black},
    urlcolor={blue!60!black}
}

\usepackage{amsfonts,amsthm,amssymb,amsmath}
\usepackage[T1]{fontenc}
\usepackage{microtype}
\usepackage{graphicx}
\usepackage{tikz}
\usepackage[usenames,dvipsnames]{pstricks}
\usepackage{epsfig}
\usepackage{enumitem}
\setlist[enumerate]{label=(\arabic*), ref=(\arabic*)}
\usepackage{verbatim}
\usepackage{thmtools, thm-restate}
\usepackage{wasysym}
\renewcommand{\square}{\mathbin\Box}

\newtheorem{theorem}{Theorem}[section]
\newtheorem{lemma}[theorem]{Lemma}
\newtheorem{proposition}[theorem]{Proposition}
\newtheorem{corollary}[theorem]{Corollary}

\newtheorem{claim}[theorem]{Claim}
\newtheorem*{ubqconjecture}{The Ubiquity Conjecture}

\theoremstyle{definition}
\newtheorem{definition}[theorem]{Definition}
\newtheorem{remark}[theorem]{Remark}

\newtheorem{cordef}[theorem]{Corollary and Definition}

\newcommand{\Half}{\mathbb{H}}
\newcommand{\Z}{\mathbb{Z}}
\newcommand{\N}{\mathbb{N}}

\newcommand{\Nbb}{\mathbb{N}}
\newcommand{\Pcal}{{\mathcal P}}
\newcommand{\Qcal}{{\mathcal Q}}
\newcommand{\Rcal}{{\mathcal R}}
\newcommand{\Scal}{{\mathcal S}}
\newcommand{\Tcal}{{\mathcal T}}
\newcommand{\Fcal}{{\mathcal F}}
\renewcommand{\triangleleft}{\vartriangleleft}
\renewcommand{\leq}{\leqslant}
\renewcommand{\preceq}{\preccurlyeq}

\newcommand\subsetsim{\mathrel{%
  \ooalign{\raise0.2ex\hbox{$\subset$}\cr\hidewidth\raise-0.8ex\hbox{\scalebox{0.9}{$\sim$}}\hidewidth\cr}}}

\newcommand{\Gtribe}{{$G$-\text{tribe}}}
\newcommand{\Gsubtribe}{{$G$-\text{subtribe}}}

\DeclareMathOperator{\RG}{RG}

\title[Ubiquity of graphs with nowhere-linear end structure]{Ubiquity in graphs II: Ubiquity of graphs with nowhere-linear end structure}

\author[Bowler, Elbracht, Erde, Gollin, Heuer, Pitz, Teegen]{Nathan Bowler \and Christian Elbracht \and Joshua Erde \and J.~Pascal Gollin \and Karl Heuer \and Max Pitz \and Maximilian Teegen}

\address[Bowler, Elbracht, Pitz, Teegen]{Universit\"{a}t Hamburg, Department of Mathematics, Bundesstra{\ss}e 55 (Geomatikum), 20146 Hamburg, Germany}

\address[Erde]{Graz University of Technology, Institute of Discrete Mathematics, Steyrergasse 30, 8010 Graz, Austria}

\address[Gollin]{Institute for Basic Science (IBS), Discrete Mathematics Group, 55, Expo-ro, Yuseong-gu, Daejeon, Republic of Korea, 
34126}

\address[Heuer]{Technische Universit\"{a}t Berlin,
Institut f\"{u}r Softwaretechnik und Theoretische Informatik,
Ernst-Reuter-Platz 7, 10587 Berlin, Germany}

\email{nathan.bowler@uni-hamburg.de}
\email{christian.elbracht@uni-hamburg.de}
\email{erde@math.tugraz.at}
\email{pascalgollin@ibs.re.kr}
\email{karl.heuer@tu-berlin.de}
\email{max.pitz@uni-hamburg.de}
\email{maximilian.teegen@uni-hamburg.de}

\thanks{The fourth author was supported by the Institute for Basic Science (IBS-R029-C1).}
\thanks{The fifth author was supported by the European Research Council (ERC) under the European Union's Horizon 2020 research and innovation programme (ERC consolidator grant DISTRUCT, agreement No.\ 648527).}

\begin{document}

\begin{abstract}
    A graph~$G$ is said to be \emph{$\preceq$-ubiquitous}, where~$\preceq$ is the minor relation between graphs, if whenever~$\Gamma$ is a graph with ${nG \preceq \Gamma}$ for all~${n \in \mathbb{N}}$, then one also has ${\aleph_0 G \preceq \Gamma}$, where~${\alpha G}$ is the disjoint union of~$\alpha$ many copies of~$G$. 
    A well-known conjecture of Andreae is that every locally finite connected graph is $\preceq$-ubiquitous.
    
    In this paper we give a sufficient condition on the structure of the ends of a graph~$G$ which implies that~$G$ is $\preceq$-ubiquitous. 
    In particular this implies that the full-grid is $\preceq$-ubiquitous.
\end{abstract}

\maketitle
\date{}

\section{Introduction}

This paper is the second in a series of papers making progress towards a conjecture of Andreae on the \emph{ubiquity} of graphs. 
Given a graph~$G$ and some relation~$\triangleleft$ between graphs we say that~$G$ is \emph{$\triangleleft$-ubiquitous} if whenever~$\Gamma$ is a graph such that ${nG \triangleleft \Gamma}$ for all~${n \in \mathbb{N}}$, then ${\aleph_0 G \triangleleft \Gamma}$, where~${\alpha G}$ denotes the disjoint union of~$\alpha$ many copies of~$G$. 
For example, a classic result of Halin~\cite{H65} says that the ray is $\subseteq$-ubiquitous, where~$\subseteq$ is the subgraph relation. 

Examples of graphs which are not ubiquitous with respect to the subgraph or topological minor relation are known (see~\cite{A13} for some particularly simple examples). 
In~\cite{A02} Andreae initiated the study of ubiquity of graphs with respect to the minor relation~$\preceq$. 
He constructed a graph which is not $\preceq$-ubiquitous, however the construction relied on the existence of a counterexample to the well-quasi-ordering of infinite graphs under the minor relation, for which only examples of size at least the continuum are known~\cite{T88}. 
In particular, the question of whether there exists a countable graph which is not $\preceq$-ubiquitous remains open. 
Most importantly, however, Andreae~\cite{A02} conjectured that at least all locally finite graphs, those with all degrees finite, should be $\preceq$-ubiquitous.

\begin{ubqconjecture}
    Every locally finite connected graph is $\preceq$-ubiquitous.
\end{ubqconjecture}

In~\cite{A13} Andreae proved that his conjecture holds for a large class of locally finite graphs. 
The exact definition of this class is technical, but in particular his result implies the following.
\begin{theorem}[Andreae, {\cite[Corollary 2]{A13}}]
    \label{t:And2}
    Let~$G$ be a connected, locally finite graph of finite tree-width such that every block of~$G$ is finite. 
    Then~$G$ is $\preceq$-ubiquitous.
\end{theorem}
\footnotetext[1]{A precise definitions of rays, the ends of a graph, their degree, and what it means for a ray to converge to an end can be found in Section~\ref{s:prelim}.}
Note that every end in such a graph must have degree\footnotemark[1] one. 

Andreae's proof employs deep results about well-quasi-orderings of labelled (infinite) trees~\cite{L71}. 
Interestingly, the way these tools are used does not require the extra condition in Theorem~\ref{t:And2} that every block of~$G$ is finite and so it is natural to ask if his proof can be adapted to remove this condition. 
And indeed, it is the purpose of the present and subsequent paper in our series,~\cite{BEEGHPTIII}, to show that this is possible, i.e.\ that all connected, locally finite graphs of finite tree-width are $\preceq$-ubiquitous.

\begin{figure}[ht]
    \center
    \begin{tikzpicture}[scale=.3]
        \tikzstyle{ray}=[->,thick]
\node[blue, scale=2] at (-8,3) {$\mathcal{R}$};
\draw[->,thick,blue!30] (0,0)--(14,0);
\draw[->,thick,blue!30] (0,2)--(14,2);
\draw[->,thick,blue!30] (0,4)--(14,4);
\draw[->,thick,blue!30] (0,6)--(14,6);
\draw[very thick,blue] (0,0)--(-6,0);
\draw[very thick,blue] (0,2)--(-6,2);
\draw[very thick,blue] (0,4)--(-6,4);
\draw[very thick,blue] (0,6)--(-6,6);

\node[red!90!black, scale=2] at (7,14) {$\mathcal{S}$};
\draw[->,very thick,red!90!black] (4,8)--(4,12);
\draw[->,very thick,red!90!black] (6,8)--(6,12);
\draw[->,very thick,red!90!black] (8,8)--(8,12);
\draw[->,very thick,red!90!black] (10,8)--(10,12);
\draw[thick,red!30] (4,-4)--(4,8);
\draw[thick,red!30] (6,-4)--(6,8);
\draw[thick,red!30] (8,-4)--(8,8);
\draw[thick,red!30] (10,-4)--(10,8);

\draw [very thick, green!60!black] plot [smooth, tension=1.5] coordinates {(0,6) (2,4) (4,-2) (10,8)};
\draw [very thick, green!60!black] plot [smooth, tension=2] coordinates {(0,4) (3,4) (6,2) (8,8)};
\draw [very thick, green!60!black] plot [smooth, tension=2] coordinates {(0,2) (3,4) (4,8)};
\draw [very thick, green!60!black] plot [smooth, tension=2] coordinates {(0,0) (6,-2) (10,0) (6,8)};

\node[green!60!black, scale=2] at (7,-6) {$\mathcal{P}$};
    \end{tikzpicture}
    \caption{A linkage between~$\Rcal$ and~$\Scal$.}
    \label{fig:linkage}
\end{figure}

The present paper lays the groundwork for this extension of Andreae's result. 
The fundamental obstacle one encounters when trying to extend Andreae's methods is the following: 
In the proof we often have two families of disjoint rays ${\mathcal{R} = (R_i \colon i \in  I)}$ and ${\mathcal{S} = (S_j \colon j \in J)}$ in~$\Gamma$, which we may assume all converge\footnotemark[1] to a common end of~$\Gamma$, and we wish to find a \emph{linkage} between~$\mathcal{R}$ and~$\mathcal{S}$, that is, an injective function ${\sigma \colon I \rightarrow J}$ and a set~$\mathcal{P}$ of disjoint finite paths~$P_i$ from~${x_i \in R_i}$ to~${y_{\sigma(i)} \in S_{\sigma(i)}}$ such that the walks
\[
    \mathcal{T} = (R_ix_iP_iy_{\sigma(i)}S_{\sigma(i)} \colon i \in I)
\]
formed by following each~$R_i$ along to~$x_i$, then following the path~$P_i$ to~$y_{\sigma(i)}$, then following the tail of~$S_{\sigma(i)}$, form a family of disjoint rays (see Figure~\ref{fig:linkage}). 
Broadly, we can think of this as `re-routing' the rays~$\mathcal{R}$ to some subset of the rays in~$\mathcal{S}$. 
Since all the rays in~$\mathcal{R}$ and~$\mathcal{S}$ converge to the same end of~$\Gamma$, it is relatively simple to show that, as long as~${|I| \leq |J|}$, there is enough connectivity between the rays in~$\Gamma$ to ensure that such a linkage always exists. 

However, in practice it is not enough for us to be guaranteed the existence of some injection~$\sigma$ giving rise to a linkage, but instead we want to choose~$\sigma$ in advance, and be able to find a corresponding linkage afterwards. 

In general, however, it is possible that for certain choices of $\sigma$ no suitable linkage exists. 
Consider, for example, the case where~$\Gamma$ is the \emph{half-grid} (which we denote by ${\mathbb{Z} \square \mathbb{N}}$), which is the graph whose vertex set is~${\mathbb{Z} \times \mathbb{N}}$ and where two vertices are adjacent if they differ in precisely one co-ordinate and the difference in that co-ordinate is one. 
If we consider two sufficiently large families of disjoint rays~$\mathcal{R}$ and~$\mathcal{S}$ in~$\Gamma$, then it is not hard to see that both~$\mathcal{R}$ and~$\mathcal{S}$ inherit a linear ordering from the planar structure of~$\Gamma$, which must be preserved by any linkage between them. 

By analysing the possible kind of linkages which can arise between two families of rays converging to a given end, we will give a classification of ends of infinite degree, which we call \emph{thick}, into three different types depending on the possible linkages they support. 
Roughly all such ends will either be \emph{pebbly}, meaning that we can always find suitable linkages for all~$\sigma$ as above, \emph{half-grid-like}, and exhibit behaviour similar to to that of the half-grid~${\mathbb{Z} \square \mathbb{N}}$, or \emph{grid-like}, and exhibit behaviour similar to to that of the full-grid~${\mathbb{Z} \square \mathbb{Z}}$ (which is analogously defined as the half-grid but with~${\mathbb{Z} \times \mathbb{Z}}$ as vertex set). 
We will give precise definitions of these terms in Sections~\ref{s:pebblyend} and~\ref{s:gridandhalfgrid}.

\begin{theorem}
    \label{t:classification}
    Let~$\Gamma$ be a graph and let~$\epsilon$ be a thick end of~$\Gamma$. 
    Then~$\epsilon$ is either pebbly, half-grid-like or grid-like.
\end{theorem}

If appropriate ends of~$\Gamma$ are pebbly, then this freedom in choosing our linkages would allow us to follow Andreae's proof strategy in order to prove the ubiquity of~$G$. However, in fact the property of an end being pebbly is so strong that we do not need to follow Andreae's strategy for such graphs. More precisely, in an pebbly end we can use the existence of such linkages to directly build a $K_{\aleph_0}$-minor of~$\Gamma$ (See Lemma~\ref{l:pebblycompleteminor}), from which it follows that~${\aleph_0G \preceq \Gamma}$ for any countable graph $G$. In this way, Theorem~\ref{t:classification} can be thought of as a local structure theorem for the ends of a graph which don't contain a $K_{\aleph_0}$-minor. 

In this way, Theorem~\ref{t:classification} allows us to make some structural assumptions on the `host' graph $\Gamma$ when considering the question of~$\preceq$-ubiquity. However, more importantly, it also allows us to make some structural assumptions about~$G$. Roughly, if the ends of~$G$ do not have a particularly simple structure then the fact that~${nG \preceq \Gamma}$ for each~${n \in \Nbb}$ will imply that~$\Gamma$ must have a pebbly end.

Analysing this situation gives rise to the following definition: 
We say that an end~$\epsilon$ of a graph~$G$ is \emph{linear} if for every finite set~$\mathcal{R}$ of at least three disjoint rays in~$G$ which converge to~$\epsilon$ we can order the elements of~$\mathcal{R}$ as ${\mathcal{R} = \{R_1,R_2,\ldots,R_n\}}$ such that for each~${1 \leq k < i < \ell \leq n}$, the rays~$R_k$ and~$R_\ell$ belong to different ends of~${G - V(R_i)}$. 

For example, the half-grid has a unique end and it is linear. 
On the other end of the spectrum, let us say that a graph~$G$ has \emph{nowhere-linear end structure} if no end of~$G$ is linear. 

Our main theorem in this paper is the following.

\begin{restatable}{theorem}{nonlin}
    \label{t:nonlin}
    Every locally finite connected graph with nowhere-linear end structure is $\preceq$-ubiquitous.
\end{restatable}

More generally, these ideas will allow us to assume, when following the proof strategy of Andreae, that certain ends of~$\Gamma$ are grid-like or half-grid-like, and that certain ends of~$G$ are linear. 
The fact that~$G$ is linear will mean that the only functions~$\sigma$ that we have to consider are ones which preserve the linear ordering on the rays, and the fact that~$\Gamma$ is grid- or half-grid-like will allow us to deduce that appropriate linkages exist for such functions. 
This will be a key part of our extension of Theorem~\ref{t:And2} in~\cite{BEEGHPTIII}.

However, independently of these potential later developments, our methods already allow us to establish new ubiquity results for many natural graphs and graph classes.

As a first concrete example, consider the full-grid ${G = \mathbb{Z} \square \mathbb{Z}}$. 
$G$ is one-ended, and for any ray~$R$ in~$G$, the graph~${G - V(R)}$ still has at most one end. 
Hence the unique end of~$G$ is non-linear, and so Theorem~\ref{t:nonlin} has the following corollary: 

\begin{restatable}{corollary}{fullgrid}
    \label{c:fullgrid}
    The full-grid is $\preceq$-ubiquitous.
\end{restatable}

Using an argument similar in spirit to that of Halin~\cite{halin1975problem}, we also establish the following theorem in this paper:

\begin{restatable}{theorem}{halfgrid}
    \label{t:halfgrid}
    Any connected minor of the half-grid ${\mathbb{N} \square \mathbb{Z}}$ is $\preceq$-ubiquitous. 
\end{restatable}

Since every countable tree is a minor of the half-grid, Theorem~\ref{t:halfgrid} implies that all countable trees are $\preceq$-ubiquitous, see Corollary~\ref{cor_trees}. 
We remark that while it has been shown that all trees are ubiquitous with respect to the topological minor relation,~\cite{BEEGHPTI}, the question of whether all uncountable trees are $\preceq$-ubiquitous has remains open, and we hope to resolve this in a paper in preparation. 

In a different direction, if~$G$ is any locally finite connected graph, then it is possible to show that~${G \square \mathbb{Z}}$ or~${G \square \mathbb{N}}$ either have nowhere-linear end structure, or are either the full-grid or a subgraph of the half-grid. Hence, Theorems~\ref{t:nonlin} and~\ref{t:halfgrid} and Corollary~\ref{c:fullgrid} have the following corollary.

\begin{restatable}{theorem}{gridprod}
    \label{t:gridprod}
    For every locally finite connected graph~$G$, both~${G \square \mathbb{Z}}$ and~${G \square \mathbb{N}}$ are $\preceq$-ubiquitous.
\end{restatable}

Finally, we will also show the following result about non-locally finite graphs. 
For~${k \in \mathbb{N}}$, we let the \emph{$k$-fold dominated ray} be the graph~$DR_k$ formed by taking a ray together with~$k$ additional vertices, each of which we make adjacent to every vertex in the ray. 
For~${k \leq 2}$, $DR_k$ is a minor of the half-grid, and so ubiquitous by Theorem~\ref{t:halfgrid}. 
In our last theorem, we show that~$DR_k$ is ubiquitous for all~${k \in \mathbb{N}}$.

\begin{restatable}{theorem}{dominatedray}
    \label{t:dominatedray}
    The $k$-fold dominated ray~$DR_k$ is $\preceq$-ubiquitous for every~${k \in \mathbb{N}}$.
\end{restatable}

The paper is structured as follows: 
In Section~\ref{s:prelim} we introduce some basic terminology for talking about minors. 
In Section~\ref{s:raygraph} we introduce the concept of a \emph{ray graph} and \emph{linkages} between families of rays, which will help us to describe the structure of an end. 
In Sections~\ref{s:pebblegame} and~\ref{s:pebblyend} we introduce a pebble-pushing game which encodes possible linkages between families of rays and use this to give a sufficient condition for an end to contain a countable clique minor. 
In Sections~\ref{s:nonpebbly} and~\ref{s:gridandhalfgrid} we prove Theorem~\ref{t:classification}, classifying the thick ends which are non-pebbly. 
In Section~\ref{s:tribes} we re-introduce some concepts from~\cite{BEEGHPTI} and show that we may assume that the $G$-minors in~$\Gamma$ are \emph{concentrated} towards some end~$\epsilon$ of~$\Gamma$. 
In Section~\ref{s:halfgrid} we use the results of the previous section to prove Theorem~\ref{t:halfgrid} and finally in Section~\ref{s:nonlinear} we prove Theorem~\ref{t:nonlin} and its corollaries.

\section{Preliminaries}
\label{s:prelim}

In our graph theoretic notation we generally follow the textbook of Diestel~\cite{D16}. 
Given two graphs~$G$ and~$H$ the \emph{cartesian product} ${G \square H}$ is a graph with vertex set~${V(G) \times V(H)}$ with an edge between~${(a,b)}$ and~${(c,d)}$ if and only if~${a = c}$ and~${bd \in E(H)}$ or~${ac \in E(G)}$ and~${b = d}$.

\begin{definition}
	A one-way infinite path is called a \emph{ray} and a two-way infinite path is called a \emph{double ray}.
    
    For a path or ray~$P$ and vertices~${v,w \in V(P)}$, let~${vPw}$ denote the subpath of~$P$ with endvertices~$v$ and~$w$. 
    If~$P$ is a ray, let~$Pv$ denote the finite subpath of~$P$ between the initial vertex of~$P$ and~$v$, 
    and let~$vP$ denote the subray (or \emph{tail}) of~$P$ with initial vertex~$v$.
    
    Given two paths or rays~$P$ and~$Q$ which are disjoint but for one of their endvertices, we write~$PQ$ for the \emph{concatenation of~$P$ and~$Q$}, that is the path, ray or double ray~${P \cup Q}$. 
    Moreover, if we concatenate paths of the form~$vPw$ and~$wQx$, then we omit writing~$w$ twice and denote the concatenation by~$vPwQx$. 
\end{definition}

\begin{definition}[{Ends of a graph, cf.~\cite[Chapter~8]{D16}}]
    An \emph{end} of an infinite graph~$\Gamma$ is an equivalence class of rays, where two rays~$R$ and~$S$ are equivalent if and only if there are infinitely many vertex disjoint paths between~$R$ and~$S$ in~$\Gamma$. 
    We denote by~${\Omega(\Gamma)}$ the set of ends of~$\Gamma$.
    
    We say that a ray~${R \subseteq \Gamma}$ \emph{converges} (or \emph{tends}) to an end~$\epsilon$ of~$\Gamma$ if~$R$ is contained in $\epsilon$. 
    In this case we call~$R$ an \emph{$\epsilon$-ray}.
    
    Given an end~${\epsilon \in \Omega(\Gamma)}$ and a finite set~${X \subseteq V(\Gamma)}$ there is a unique component of~${\Gamma - X}$ which contains a tail of every ray in~$\epsilon$, which we denote by~${C(X,\epsilon)}$. 
    
    For an end~${\epsilon \in \Omega(\Gamma)}$ we define the \emph{degree} of~$\epsilon$ in~$\Gamma$ as the supremum in ${\Nbb \cup \{ \infty \}}$ of the set ${ \{ |\mathcal{R}| \; \colon \; \mathcal{R} \textnormal{ is a set of disjoint } \epsilon \textnormal{-rays} \} }$. 
    Note that this supremum is in fact an attained maximum, i.e.~for each end~$\epsilon$ of~$\Gamma$ there is a set~$\mathcal{R}$ of vertex-disjoint $\epsilon$-rays with~${|\mathcal{R}| = \deg(\omega)}$, as proved by Halin~\cite[Satz~1]{H65}. 
    If an end has finite degree, we call it \emph{thin}. Otherwise, we call it \emph{thick}.
    
    A vertex~${v \in V(\Gamma)}$ \emph{dominates} an end~${\epsilon \in \Omega(\Gamma)}$ if there is a ray~${R \in \omega}$ such that there are infinitely many $v$\,--\,$R$\,-paths in~$\Gamma$ that are vertex disjoint apart from~$v$. 
\end{definition}

We will use the following two basic facts about infinite graphs.

\begin{proposition}
    \label{p:rayorinfvertex}\cite[Proposition~8.2.1]{D16}
    An infinite connected graph contains either a ray or a vertex of infinite degree.
\end{proposition}

\begin{proposition}
    \label{p:infdomTKaleph}\cite[Exercise 8.19]{D16}
    A graph~$G$ contains a subdivided~$K_{\aleph_0}$ as a subgraph if and only if~$G$ has an end which is dominated by infinitely many vertices.
\end{proposition}

\begin{definition}[Inflated graph, branch set]
    Given a graph~$G$ we say that a pair~${(H,\varphi)}$ is an \emph{inflated copy of~$G$}, or an~$IG$, if~$H$ is a graph and~${\varphi \colon V(H) \rightarrow V(G)}$ is a map such that:
    \begin{itemize}
        \item For every~${v \in V(G)}$ the \emph{branch set~${\varphi^{-1}(v)}$} induces a non-empty, connected subgraph of~$H$;
        \item There is an edge in~$H$ between~${\varphi^{-1}(v)}$ and~${\varphi^{-1}(w)}$ if and only if~${vw \in E(G)}$ and this edge, if it exists, is unique.
    \end{itemize}
\end{definition}

When there is no danger of confusion we will simply say that~$H$ is an~$IG$ instead of saying that~${(H,\varphi)}$ is an~$IG$, and denote by~${H(v)=\varphi^{-1}(v)}$ the branch set of~$v$. 

\begin{definition}[Minor]
    A graph~$G$ is a minor of another graph~$\Gamma$, written~${G \preceq \Gamma}$, if there is some subgraph~${H \subseteq \Gamma}$ such that~$H$ is an inflated copy of~$G$. 
\end{definition}

\begin{definition}[Extension of inflated copies]
    Suppose~${G \subseteq G'}$ as subgraphs, and that~$H$ is an~$IG$ and~$H'$ is an~$IG'$. 
    We say that~$H'$ \emph{extends}~$H$ (or that~$H'$ is an extension of~$H$) 
    if~${H \subseteq H'}$ as subgraphs and~${H(v) \subseteq H'(v)}$ for all~${v \in V(G) \cap V(G')}$. 
\end{definition}

Note that since~${H \subseteq H'}$, for every edge~${vw \in E(G)}$, the unique edge between the branch sets~${H'(v)}$ and~${H'(w)}$ is also the unique edge between~${H(v)}$ and~${H(w)}$.

\begin{definition}[Tidiness]
    Let~${(H,\varphi)}$ be an~$IG$. 
    We call~${(H,\varphi)}$ \emph{tidy} if 
    \begin{itemize}
        \item $H[\varphi^{-1}(v)]$ is a tree for all~${v \in V(G)}$;
        \item $H[\varphi^{-1}(v)]$ is finite if~$d_G(v)$ is finite.
    \end{itemize}
\end{definition}

Note that every~$H$ which is an~$IG$ contains a subgraph~$H'$ such that~${(H',\varphi \restriction V(H'))}$ is a tidy~$IG$, although this choice may not be unique. 
In this paper we will always assume without loss of generality that each~$IG$ is tidy.

\begin{definition}[Restriction]
    Let~$G$ be a graph, ${M \subseteq G}$ a subgraph of~$G$, and let~${(H,\varphi)}$ be an~$IG$. 
    The \emph{restriction of~$H$ to~$M$}, denoted by~${H(M)}$, is the~$IM$ given by~${(H(M),\varphi')}$ where~${\varphi'^{-1}(v) = \varphi^{-1}(v)}$ for all~${v \in V(M)}$ and~${H(M)}$ consists of union of the subgraphs of~$H$ induced on each branch set~${\varphi^{-1}(v)}$ for each~${v \in V(M)}$ together with the edge between~${\varphi^{-1}(u)}$ and~${\varphi^{-1}(v)}$ for each~${(u,v) \in E(M)}$. 
\end{definition}

Suppose~$R$ is a ray in some graph~$G$. 
If~$H$ is a tidy~$IG$ in a graph~$\Gamma$ then in the restriction~$H(R)$ all rays which do not have a tail contained in some branch set will share a tail. 
Later in the paper we will want to make this correspondence between rays in~$G$ and~$\Gamma$ more explicit, with use of the following definition: 

\begin{definition}[Pullback]
    Let~$G$ be a graph, ${R \subseteq G}$ a ray, and let~$H$ be a tidy~$IG$. 
    The \emph{pullback of~$R$ to~$H$} is the subgraph~${H^{\downarrow}(R) \subseteq H}$ where~$H^{\downarrow}(R)$ is subgraph minimal such that~${(H^{\downarrow}(R), \varphi \restriction V(H^{\downarrow}(R)))}$ is an~$IM$.
\end{definition}

Note that, since~$H$ is tidy, $H^{\downarrow}(R)$ is well defined. 
As we shall see, $H^{\downarrow}(R)$ will be a ray.

\begin{lemma}
    \label{l:pullbackray}
    Let~$G$ be a graph and let~$H$ be a tidy~$IG$. 
    If~${R \subseteq G}$ is a ray, then the pullback~${H^{\downarrow}(R)}$ is also a ray.
\end{lemma}

\begin{proof}
    Let~${R = x_1x_2\ldots}$. 
    For each integer~${i \geq 1}$ there is a unique edge~${v_i w_i \in E(H)}$ between the branch sets~$H(x_i)$ and~$H(x_{i+1})$. 
    By the tidiness assumption, $H(x_{i+1})$ induces a tree in~$H$, and so there is a unique path~${P_i \subset H(x_{i+1})}$ from~$w_i$ to~$v_{i+1}$ in~$H$. 
    
    By minimality of~$H^{\downarrow}(R)$, it follows that~${H^{\downarrow}(R)(x_1) = \{v_1\}}$ and~${H^{\downarrow}(R)(x_{i+1}) = V(P_i)}$ for each~${i \geq 1}$. 
    Hence~${H^{\downarrow}(R)}$ is a ray.
\end{proof}

\section{The Ray Graph}
\label{s:raygraph}

\begin{definition}[Ray graph]
    Given a finite family of disjoint rays~${\mathcal{R} = (R_i \colon i \in I)}$ in a graph~$\Gamma$ 
    the \emph{ray graph} ${\RG_{\Gamma}(\mathcal{R})=\RG_{\Gamma}(R_i \colon i \in I)}$ is the graph with vertex set~$I$ and with an edge between~$i$ and~$j$ if there is an infinite collection of vertex disjoint paths from~$R_i$ to~$R_j$ in~$\Gamma$ which meet no other~$R_k$. 
    When the host graph~$\Gamma$ is clear from the context we will simply write~${\RG(\mathcal{R})}$ for~${\RG_{\Gamma}(\mathcal{R})}$.
\end{definition}

The following lemmas are simple exercises. 
For a family~$\mathcal{R}$ of disjoint rays in~$G$ tending to the same end and~${H \subseteq \Gamma}$ being an~$IG$ the aim is to establish the following: 
if~$\mathcal{S}$ is a family of disjoint rays in~$\Gamma$ which contains the pullback~$H^{\downarrow}(R)$ of each~${R \in \mathcal{R}}$, then the subgraph of the ray graph~${\RG_{\Gamma}(\mathcal{S})}$ induced on the vertices given by~${\{H^{\downarrow}(R) \, : \, R \in \mathcal{R} \}}$ is connected.

\begin{lemma}
    \label{l:connraygraph}
    Let~$G$ be a graph and let~${\mathcal{R} = (R_i \colon i \in I)}$ be a finite family of disjoint rays in~$G$. Then~$\RG_G(\mathcal{R})$ is connected if and only if all rays in~$\mathcal{R}$ tend to a common end~${\omega \in \Omega(G)}$. 
    Moreover, if~$R'_i$ is a tail of~$R_i$ for each~${i \in I}$, then we have that ${{\RG( R_i \colon i \in I)} = {\RG( R'_i \colon i \in I)}}$. 
\end{lemma}

\begin{lemma}
    \label{l:pullbackraygraph}
    Let~$G$ be a graph, ${\mathcal{R} = (R_i \colon i \in I)}$ be a finite family of disjoint rays in~$G$ and let~$H$ be an~$IG$. 
    If ${\mathcal{R'} = (H^{\downarrow}(R_i) \colon i\in I)}$ is the set of pullbacks of the rays in~$\mathcal{R}$ in~$H$, then~${\RG_G(\mathcal{R})= \RG_H(\mathcal{R}')}$.
\end{lemma}

\begin{lemma}
    \label{l:rayinducedsubgraph}
    Let~$G$ be a graph, ${H \subseteq G}$, ${\mathcal{R} = (R_i \colon i \in I)}$ be a finite disjoint family of rays in~$H$ and let ${\mathcal{S} = (S_j \colon j \in J)}$ be a finite disjoint family of rays in~${G - V(H)}$, where~$I$ and~$J$ are disjoint. 
    Then~$\RG_H(\mathcal{R})$ is a subgraph of~${\RG_G(\mathcal{R} \cup \mathcal{S})\big[I \big]}$. In particular, if all rays in~$\mathcal{R}$ tend to a common end in~$H$, then~${\RG_G(\mathcal{R} \cup \mathcal{S})\big[I \big]}$ is connected.
\end{lemma}

Recall that an end~$\omega$ of a graph~$G$ is called \emph{linear} if for every finite set~$\mathcal{R}$ of at least three disjoint $\omega$-rays in~$G$ we can order the elements of~$\mathcal{R}$ as~${\mathcal{R} = \{R_1,R_2,\ldots,R_n\}}$ 
such that for each~${1 \leq k < i < \ell \leq n}$, the rays~$R_k$ and~$R_\ell$ belong to different ends of~${G - V(R_i)}$. 

\begin{lemma}
    \label{l:linearraygraph}
    An end~$\omega$ of a graph~$G$ is linear if and only if the ray graph of every finite family of disjoint $\omega$-rays is a path.
\end{lemma}

\begin{proof}
    For the forward direction suppose~$\omega$ is linear and~${\{ R_1,R_2, \ldots, R_n \}}$ converge to~$\omega$, with the order given by the definition of linear. 
    It follows that there is no~${1 \leq k < i < \ell \leq n}$ such that~${k \ell}$ is an edge in~${\RG(R_j \colon j \in [n])}$. 
    However, by Lemma~\ref{l:connraygraph} ${\RG(R_j \colon j \in [n])}$ is connected, and hence it must be the path~${12\ldots n}$. 
    
    Conversely, suppose that the ray graph of every finite family of $\omega$-rays is a path. 
    Then, every such family~$\mathcal{R}$ can be ordered as~${\{R_1,R_2,\ldots, R_n \}}$ such that~$\RG(\mathcal{R})$ is the path~${12 \ldots n}$. 
    In particular, for each~$i$, we have that ${k \ell \not \in E(\RG(\mathcal{R}))}$ whenever ${1 \leq k < i < \ell \leq n-1}$.
    
    Suppose for a contradiction that there exists~${1 \leq k < i < \ell \leq n-1}$ such that~$R_k$ and~$R_\ell$ belong to the same end of~${G - V(R_i)}$, and so there is an infinite family of vertex disjoint paths~$\mathcal{P}$ from~$R_k$ to~$R_{\ell}$ in~${G - V(R_i)}$. 
    Each of these paths must contain a subpath which goes from a ray~$R_r$ for some~${1 \leq r < i}$ to a ray~$R_s$ for some~${i < s \leq n-1}$, and which meets no other ray in~$\Rcal$. 
    Since there are infinitely many paths, by the pigeon hole principle there is some~${1 \leq r < i <s \leq n-1}$ such that there are infinitely many vertex disjoint paths from~$R_r$ to~$R_s$ in~${G \setminus V(R_i)}$ which meet not other ray in~$\Rcal$, and so~${rs \in  E(\RG(\mathcal{R}))}$, a contradiction.
\end{proof}

We will also use the following lemma, whose proof is an easy exercise.

\begin{lemma}
    \label{l:raygraph_subfamily}
    Let ${\mathcal{R} = (R_i\colon i\in I)}$ be a finite family of disjoint rays in~$G$ and let \linebreak ${\Rcal'=(R_i\colon i\in J)}$ be a subfamily of~$\Rcal$. 
    Then~${\RG(\Rcal')}$ contains an edge between~${i \in J}$ and~${j \in J}$ if and only if~$i$ and~$j$ lie in the same component of~${\RG(\Rcal)-(J \setminus \{i,j\})}$.
\end{lemma}


\begin{definition}[Tail of a ray after a set]
    Given a ray~$R$ in a graph~$G$ and a finite set~${X \subseteq V(G)}$ the \emph{tail of~$R$ after~$X$}, denoted by~${T(R,X)}$, is the unique infinite component of~$R$ in~${G- X}$.
\end{definition}

\begin{definition}[Linkage of families of rays]
    \label{d:linkage}
    Let~${\mathcal{R} = (R_i \colon i \in I)}$ and~${\mathcal{S} = (S_j \colon j \in J)}$ be families of disjoint rays of~$G$, where the initial vertex of each~$R_i$ is denoted~$x_i$. 
    A family~${\mathcal{P} = (P_i \colon i \in I)}$ of paths in~$G$ is a \emph{linkage} from~$\mathcal{R}$ to~$\mathcal{S}$ if there is an injective function~${\sigma \colon I \rightarrow J}$ such that
    \begin{itemize}
        \item Each~$P_i$ goes from a vertex~${x'_i \in R_i}$ to a vertex~${y_{\sigma(i)} \in S_{\sigma(i)}}$;
        \item The family ${\mathcal{T} = (x_iR_ix'_iP_iy_{\sigma(i)}S_{\sigma(i)} \colon i\in I)}$ is a collection of disjoint rays.
    \end{itemize}
    We say that~$\Tcal$ is obtained by {\em transitioning} from~$\Rcal$ to~$\Scal$ along the linkage.
    We say the linkage~$\mathcal{P}$ \emph{induces} the mapping~$\sigma$. 
    Given a vertex set~${X \subseteq V(G)}$ we say that the linkage is \emph{after}~$X$ if~${X \cap V(R_i) \subseteq V(x_iR_ix'_i)}$ for all~${i \in I}$ and no other vertex in~$X$ is used by the members of~$\mathcal{T}$. 
    We say that a function~${\sigma \colon I \rightarrow J}$ is a {\em transition function} from~$\Rcal$ to~$\Scal$ if for any finite vertex set~${X \subseteq V(G)}$ there is a linkage from~$\Rcal$ to~$\Scal$ after~$X$ that induces~$\sigma$.
\end{definition}

We will need the following lemma from~\cite{BEEGHPTI}, which asserts the existence of linkages.
\begin{lemma}[Weak linking lemma {\cite[Lemma~4.3]{BEEGHPTI}}]
    \label{l:weaklink}
    Let~$G$ be a graph, ${\omega \in \Omega(G)}$ and let~${n \in \mathbb{N}}$. 
    Then for any two families~${\Rcal = (R_i \colon i \in [n])}$ and~${\Scal = (S_j \colon j \in [n])}$ of vertex disjoint $\omega$-rays and any finite vertex set~${X \subseteq V(G)}$, there is a linkage from~$\mathcal{R}$ to~$\mathcal{S}$ after~$X$.
\end{lemma}

\section{A pebble-pushing game}
\label{s:pebblegame}

Suppose we have a family of disjoint rays~${\mathcal{R} = (R_i \colon i \in I)}$ in a graph~$G$ and a subset~${J \subseteq I}$. 
Often we will be interested in which functions we can obtain as transition functions between $(R_i \colon i \in J)$ and $(R_i \colon i \in I)$. 
We can think of this as trying to `re-route' the rays $(R_i \colon i \in J)$ to the tails of a different set of $|J|$ rays in $(R_i \colon i \in I)$.

To this end, it will be useful to understand the following pebble-pushing game on a graph.

\begin{definition}[Pebble-pushing game]
    Let~${G = (V,E)}$ be a finite graph. 
    For any fixed positive integer~$k$ we call a tuple~${( x_1,x_2, \ldots, x_k) \in V^{k}}$ a \emph{game state} if~${x_i \neq x_j}$ for all~${i,j \in [k]}$ with~${i \neq j}$.
    
    The \emph{pebble-pushing game (on~$G$)} is a game played by a single player.
    Given a game state ${Y = (y_1,y_2,\ldots, y_k)}$, we imagine~$k$ labelled pebbles placed on the vertices ${(y_1,y_2,\ldots, y_k)}$.
    We move between game states by moving a pebble from a vertex to an adjacent vertex which does not contain a pebble, 
    or formally, a $Y$-\emph{move} is a game state ${Z = (z_1,z_2\ldots, z_k)}$ such that there is an ${\ell \in [k]}$ such that ${y_\ell z_\ell \in E}$ and ${y_i = z_i}$ for all~${i \in [k]\setminus\{\ell\}}$.
    
    Let ${X = (x_1,x_2\ldots, x_k)}$ be a game state.
    The \emph{$X$-pebble-pushing game (on~$G$)} is a pebble-pushing game where we start with~$k$ labelled pebbles placed on the vertices ${(x_1,x_2\ldots,x_k)}$.
    
    We say a game state~$Y$ is \emph{achievable} in the $X$-pebble-pushing game if there is a sequence $(X_i \colon i \in [n])$ of game states for some ${n \in \mathbb{N}}$ such that ${X_{1} = X}$, ${X_{n} = Y}$ and $X_{i+1}$ is an $X_{i}$-move for all~${i \in [n-1]}$, that is, if it is a sequence of moves that pushes the pebbles from~$X$ to~$Y$.
    
    A graph~$G$ is \emph{$k$-pebble-win} if~$Y$ is an achievable game state in the $X$-pebble-pushing game on~$G$ for every two game states~$X$ and~$Y$.
\end{definition}

The following lemma shows that achievable game states on the ray graph~${\RG(\mathcal{R})}$ yield transition functions from a subset of~$\Rcal$ to itself. 
Therefore, it will be useful to understand which game states are achievable, and in particular the structure of graphs on which there are unachievable game states.

\begin{lemma}
    \label{l:transitionpebble}
    Let~$\Gamma$ be a graph, ${\omega \in \Omega(\Gamma)}$, ${m \geq k}$ be positive integers and let~${(S_j \colon j \in [m])}$ be a family of disjoint rays in~$\omega$. For every achievable game state~${Z = (z_1, z_2, \ldots, z_k)}$ in the ${(1,2, \ldots, k)}$-pebble-pushing game on ${\RG(S_j \colon j \in [m])}$, the map~$\sigma$ defined via~${\sigma(i) := z_i}$ for every~${i \in [k]}$ is a transition function from~${(S_i \colon i \in [k])}$ to~${(S_j \colon j \in [m])}$.
\end{lemma}

\begin{proof}
    We first note that if~$\sigma$ is a transition function from~${(S_i \colon i \in [k])}$ to~${(S_j \colon j \in [m])}$ and~$\tau$ is a transition function from~${(S_i \colon i \in \sigma([k]))}$ to~${(S_j \colon j \in [m])}$, then clearly~${\tau \circ \sigma}$ is a transition function from~${(S_i \colon i \in [k])}$ to~${(S_j \colon j \in [m])}$.

    Hence, it is sufficient to show the statement holds when~$\sigma$ is obtained from~${(1,2, \ldots,k)}$ by a single move, that is, there is some~${t \in [k]}$ and a vertex~${\sigma(t) \not\in [k]}$ such that~$\sigma(t)$ is adjacent to~$t$ in ${\RG(S_j \colon j \in [m])}$ and~${\sigma(i) = i}$ for~${i \in [k] \setminus \lbrace t \rbrace}$.

    So, let~${X \subseteq V(G)}$ be a finite set. 
    We will show that there is a linkage from~${(S_i \colon i \in [k])}$ to~${(S_j \colon j \in [m])}$ after~$X$ that induces~$\sigma$. 
    By assumption, there is an edge~${t \sigma(t)}$ of \linebreak ${\RG(S_j \colon j \in[m])}$. 
    Hence, there is a path~$P$ between~${T(S_t, X)}$ and~${T(S_{\sigma(t)},X)}$ which avoids~$X$ and all other~$S_j$.

    Then the family~${\mathcal{P} = (P_1,P_2,\ldots,P_k)}$ where~${P_t = P}$ and~${P_i = \emptyset}$ for each~${i \neq t}$ is a linkage from~${(S_i \colon i \in [k])}$ to~${(S_j \colon j \in [m])}$ after~$X$ that induces~$\sigma$. 
\end{proof}

We note that this pebble-pushing game is sometimes known in the literature as ``permutation pebble motion''~\cite{KMS84} or ``token reconfiguration''~\cite{CDP08}. 
Previous results have mostly focused on computational questions about the game, rather than the structural questions we are interested in, but we note that in~\cite{KMS84} the authors give an algorithm that decides whether or not a graph is $k$-pebble-win, from which it should be possible to deduce the main result in this section, Lemma~\ref{l:kpebbstructure}. 
However, since a direct derivation was shorter and self contained, we will not use their results. We present the following simple lemmas without proof.

\begin{lemma}\label{l:transpebb}
    Let~$G$ be a finite graph and~$X$ a game state.
    
    \begin{itemize}
        \item If~$Y$ is an achievable game state in the $X$-pebble-pushing game on~$G$, then~$X$ is an achievable game state in the $Y$-pebble-pushing game on~$G$.
        \item If~$Y$ is an achievable game state in the $X$-pebble-pushing game on~$G$ and~$Z$ is an achievable game state in the $Y$-pebble-pushing game on~$G$, then~$Z$ is an achievable game state in the $X$-pebble-pushing game on~$G$.
    \end{itemize}
\end{lemma}

\begin{definition}
    Let~$G$ be a finite graph and let~${X = (x_1,x_2,\ldots,x_k)}$ be a game state. 
    Given a permutation~$\sigma$ of~${[k]}$ let us write~${X^{\sigma} = (x_{\sigma(1)}, x_{\sigma(2)},\ldots, x_{\sigma(k)})}$. 
    We define the \emph{pebble-permutation group} of~${(G,X)}$ to be the set of permutations~$\sigma$ of~${[k]}$ such that~$X^{\sigma}$ is an achievable game state in the $X$-pebble-pushing game on~$G$.
\end{definition}

Note that by Lemma~\ref{l:transpebb}, the pebble-permutation group of~${(G,X)}$ is a subgroup of the symmetric group~$S_k$.

\begin{lemma}
    Let~$G$ be a graph and let~$X$ be a game state. 
    If~$Y$ is an achievable game state in the $X$-pebble-pushing game and~$\sigma$ is in the pebble-permutation group of~$Y$, 
    then~$\sigma$ is in the pebble-permutation group of~$X$.
\end{lemma}

\begin{lemma}
    \label{l:permwin}
    Let~$G$ be a finite connected graph and let~$X$ be a game state. 
    Then~$G$ is $k$-pebble-win if and only if the pebble-permutation group of~${(G,X)}$ is~$S_k$.
\end{lemma}

\begin{proof}
    Clearly, if the pebble-permutation group is not~$S_k$ then~$G$ is not $k$-pebble-win. 
    Conversely, since~$G$ is connected, for any game states~$X$ and~$Y$ there is some~$\tau$ such that~$Y^{\tau}$ is an achievable game state in the $X$-pebble-pushing game, 
    since we can move the pebbles to any set of~$k$ vertices, up to some permutation of the labels. 
    We know by assumption that~$X^{\tau^{-1}}$ is an achievable game state in the $X$-pebble-pushing game. 
    Therefore, by Lemma~\ref{l:transpebb}, $Y$ is an achievable game state in the $X$-pebble-pushing game.
\end{proof}

\begin{lemma}
    \label{l:RB}
    Let~$G$ be a finite connected graph and let~${X = (x_1,x_2,\ldots,x_k)}$ be a game state. 
    If~$G$ is not $k$-pebble-win, then there is a two colouring~${c \colon X \rightarrow \{r,b\}}$ such that both colour classes are non trivial and for all~${i,j \in [k]}$ with~${c(x_i)=r}$ and~${c(x_j)=b}$ the transposition~${(ij)}$ is not in the pebble-permutation group.
\end{lemma}

\begin{proof}
    Let us draw a graph~$H$ on~${\{ x_1,x_2, \ldots, x_k \}}$ by letting~${x_ix_j}$ be an edge if and only if~${(ij)}$ is in the pebble-permutation group of~${(G,X)}$. 
    It is a simple exercise to show that the pebble-permutation group of~${(G,X)}$ is~$S_k$ if and only if~$H$ has a single component.

    Since~$G$ is not $k$-pebble-win, we know by Lemma~\ref{l:permwin} that there are at least two components in~$H$. 
    Let us pick one component~$C_1$ and set~${c(x)=r}$ for all~${x \in V(C_1)}$ and~${c(x)=b}$ for all~${x \in X \setminus V(C_1)}$. 
\end{proof}

\begin{definition}
    Given a graph~$G$, a path~${x_1 x_2 \ldots x_n}$ in~$G$ is a \emph{bare path} if~${d_G(x_i) = 2}$ for all~${2 \leq i \leq n-1}$. 
\end{definition}

\begin{lemma}
    \label{l:kpebbstructure}
    Let~$G$ be a finite connected graph with vertex set~${V := V(G)}$ which is not $k$-pebble-win and with~${|V| \geq k + 2}$. 
    Then there is a bare path ${P = p_1p_2 \ldots p_n}$ in~$G$ such that~${|V \setminus V(P)| \leq k}$. Furthermore, either every edge in~$P$ is a bridge in~$G$, or~$G$ is a cycle.
\end{lemma}

\begin{proof}
    Let~${X = (x_1,x_2,\ldots, x_k )}$ be a game state. 
    By Lemma~\ref{l:RB}, since~$G$ is not $k$-pebble-win,  there is a two colouring ${c \colon \{x_i \colon i \in [k]\} \rightarrow \{r,b\}}$ such that both colour classes are non trivial and for all~${i,j \in [k]}$ with ${c(x_i)=r}$ and ${c(x_j)=b}$ the transposition~${(ij)}$ is not in the pebble permutation group. 
    Let us consider this as a three colouring ${c\colon V \rightarrow \{r,b,0\}}$ where ${c(v) = 0}$ if ${v \not\in \{x_1,x_2, \ldots, x_k\}}$.

    For every achievable game state~${Z = (z_1,z_2,\ldots,z_k)}$ in the $X$-pebble-pushing game we define a three colouring~$c_Z$ given by~${c_Z(z_i) = c(x_i)}$ for all~${i \in [k]}$ and by~${c_Z(v) = 0}$ for all~${v \notin \lbrace z_1,z_2, \ldots, z_k \rbrace}$. 
    We note that, for any achievable game state~$Z$ there is no~${z_i \in c^{-1}_Z(r)}$ and~${z_j \in c^{-1}_Z(b)}$ such that~${(ij)}$ is in the pebble permutation group of~${(G,Z)}$. 
    Indeed, if it were, then by Lemma~\ref{l:transpebb}~$X^{(ij)}$ is an achievable game state in the $X$-pebble-pushing game, contradicting the fact that~${c(x_i) = r}$ and~${c(x_j) = b}$.

    Since~$G$ is connected, for every achievable game state~$Z$ there is a path ${P = p_1 p_2\ldots p_m}$ in~$G$ with $c_Z(p_1) =r$, $c_Z(p_m) = b$ and $c_Z(p_i)=0$ otherwise. 
    Let us consider an achievable game state~$Z$ for which~$G$ contains such a path~$P$ of maximal length.

    We first claim that there is no~${v \not\in P}$ with~${c_Z(v) = 0}$. 
    Indeed, suppose there is such a vertex~$v$. 
    Since~$G$ is connected there is some $v$--$P$ path~$Q$ in~$G$ and so, by pushing pebbles towards~$v$ on~$Q$, we can achieve a game state~$Z'$ such that ${c_{Z'} = c_{Z}}$ on~$P$ and there is a vertex~$v'$ adjacent to~$P$ such that~${c_{Z'}(v')=0}$. 
    Clearly~$v'$ cannot be adjacent to~$p_1$ or~$p_m$, since then we can push the pebble on~$p_1$ or~$p_m$ onto~$v'$ and achieve a game state~$Z''$ for which~$G$ contains a longer path than~$P$ with the required colouring. 
    However, if~$v'$ is adjacent to~$p_\ell$ with~${2 \leq \ell \leq m-1}$, then we can push the pebble on~$p_1$ onto~$p_\ell$ and then onto~$v'$, then push the pebble from~$p_m$ onto~$p_1$ and finally push the pebble on~$v'$ onto~$p_\ell$ and then onto~$p_m$.

    If~${Z' = (z'_1,z'_2,\ldots, z'_k)}$ with~${p_1 = z'_i}$ and~${p_m = z'_j}$, then above shows that~${(ij)}$ is in the pebble-permutation group of~${(G,Z')}$. 
    However, we have ${c_{Z'}(z'_i) = c_Z(p_1) = r}$ as well as ${c_{Z'}(z'_j) = c_Z(p_m) = b}$, contradicting our assumptions on~$c_{Z'}$.

    Next, we claim that each~$p_i$ with~${3 \leq i \leq m-2}$ has degree~$2$. 
    Indeed, suppose first that~$p_i$ with~${3 \leq i \leq m-2}$ is adjacent to some other~$p_j$ with~${1 \leq j \leq m}$ such that~$p_i$ and~$p_j$ are not adjacent in~$P$. 
    Then it is easy to find a sequence of moves which exchanges the pebbles on~$p_1$ and~$p_m$, contradicting our assumptions on~$c_Z$. 
    
    Suppose then that~$p_i$ is adjacent to a vertex~$v$ not in~$P$. 
    Then, ${c_Z(v) \neq 0}$, say without loss of generality~${c_Z(v) = r}$. 
    However then, we can push the pebble on~$p_m$ onto~$p_{i-1}$, push the pebble on~$v$ onto~$p_i$ and then onto~$p_m$ and finally push the pebble on~$p_{i-1}$ onto~$p_i$ and then onto~$v$. 
    As before, this contradicts our assumptions on~$c_Z$.

    Hence~${P' = p_2p_3 \ldots p_{m-1}}$ is a bare path in~$G$, and since every vertex in~${V - V(P')}$ is coloured using~$r$ or using~$b$, there are at most~$k$ such vertices.

    Finally, suppose that there is some edge in~$P'$ which is not a bridge of~$G$, and so no edge of~$P'$ is a bridge of~$G$. 
    Before we show that~$G$ is a cycle, 
    we make the following claim:

    \begin{claim}\label{c:cyclenogamestate}
        There is no achievable game state ${W=(w_1,w_2,\ldots,w_k)}$ such that there is a cycle ${C = c_1 c_2 \ldots c_r c_1}$ and a vertex~${v \not\in C}$ such that:
        \begin{itemize}
            \item There exist distinct positive integers~${i,j,s}$ and~$t$ such that~${c_W(c_i) = r}$, ${c_W(c_j) = b}$ and~${c_W(c_s)=c_W(c_t)=0}$;
            \item $v$ adjacent to some~${c_v \in C}$.
        \end{itemize}
    \end{claim}
    
    \begin{proof}[Proof of Claim~\ref{c:cyclenogamestate}]
        Suppose for a contradiction there exists such an achievable game state~$W$. 
        Since~$C$ is a cycle, we may assume without loss of generality that~${c_i = c_1}$, ${c_s = c_2 = c_v}$, ${c_t = c_3}$ and~${c_j = c_4}$. 
        If~${c_W(v) = b}$, then we can push the pebble at~$v$ to~$c_2$ and then to~$c_3$, push the pebble at~$c_1 $ to~$c_2$ and then to~$v$, and then push the pebble at~$c_3$ to~$c_1$. 
        This contradicts our assumptions on~$c_W$. 
        The case where~${c_W(v)=r}$ is similar. 
        Finally, if ${c_W(v)=0}$, then we can push the pebble at~$c_1$ to~$c_2$ and then to~$v$, then push the pebble at~$c_4$ to~$c_1$, then push the pebble at~$v$ to~$c_2$ and then to~$c_4$. 
        Again this contradicts our assumptions on~$c_W$.
    \end{proof}

    Since no edge of~$P'$ is a bridge, it follows that~$G$ contains a cycle~$C$ containing~$P'$. 
    If~$G$ is not a cycle, then there is a vertex~${v \in V \setminus C}$ which is adjacent to~$C$. 
    However by pushing the pebble on~$p_1$ onto~$p_2$ and the pebble on~$p_m$ onto~$p_{m-1}$, which is possible since~${|V| \geq k+2}$, we achieve a game state~$Z'$ such that~$C$ and~$v$ satisfy the assumptions of the above claim, a contradiction.
\end{proof}

\section{Pebbly and non-pebbly ends}
\label{s:pebblyend}

\begin{definition}[Pebbly]
    Let~$\Gamma$ be a graph and~$\omega$ an end of~$\Gamma$. 
    We say~$\omega$ is \emph{pebbly} if for every~${k \in \mathbb{N}}$ there is an~${n \geq k}$ and a family~${\mathcal{R}=(R_i \colon i \in [n])}$ of disjoint rays in~$\omega$ such that~${\RG(\mathcal{R})}$ is $k$-pebble-win. 
    If for some~$k$ there is no such family~$\mathcal{R}$, we say~$\omega$ is \emph{non-pebbly} and in particular \emph{not $k$-pebble-win}.
\end{definition}

Clearly an end of degree~$k$ is not $k$-pebble-win, since no graph on at most~$k$ vertices is $k$-pebble-win, and so every pebbly end is thick. 
However, as we shall see, pebbly ends are particularly rich in structure.

\begin{lemma}
    \label{l:pebblycompleteminor}
    Let~$\Gamma$ be a graph and let~${\omega \in \Omega(\Gamma)}$ be a pebbly end. 
    Then~${K_{\aleph_0} \preceq \Gamma}$.
\end{lemma}

\begin{proof}
By assumption, there exists a sequence $\mathcal{R}_1, \mathcal{R}_2, \dots$ of families of disjoint $\omega$-rays such that, for each $k \in \mathbb{N}$, $\RG(\mathcal{R}_k)$ is $k$-pebble-win. Let us suppose that 
\[
\mathcal{R}_i = (R^i_1, R^i_2,\ldots, R^i_{m_i}) \text{ for each } i \in \mathbb{N}.
\]

Let us enumerate the vertices and edges of $K_{\aleph_0}$ with a bijection ${\sigma \colon \mathbb{N} \cup \mathbb{N}^{(2)} \rightarrow \mathbb{N}}$ such that $\sigma(i,j) > \max \{ \sigma(i), \sigma(j)\}$ for every $\{i,j\} \in \mathbb{N}^{(2)}$ and also $\sigma(1) < \sigma(2) < \dotsb$~. For each $k \in \mathbb{N}$ let $G_k$ be the graph on vertex set $V_k = \{ i \in \mathbb{N} \, : \, \sigma(i) \leq k\}$ and edge set $E_k = \{\{i,j\} \in \mathbb{N}^{(2)} \, : \, \sigma(i,j) \leq k \}$.

We will inductively construct subgraphs $H_k$ of $\Gamma$ such that $H_k$ is an $IG_k$ extending $H_{k-1}$. Furthermore for each $k \in \mathbb{N}$ if $V(G_k) = [n]$ then there will be tails $T_1,T_2,\ldots,T_n$ of $n$ distinct rays in $\mathcal{R}_n$ such that for every $i \in [n]$ the tail $T_i$ meets $H_k$ in a vertex of the branch set of $i$, and is otherwise disjoint from $H_k$. We will assume without loss of generality that $T_i$ is a tail of $R^n_i$.

Since $\sigma(1)=1$ we can take $H_1$ to be the initial vertex of $R^1_1$. Suppose then that $V(G_{n-1}) = [r]$ and we have already constructed $H_{n-1}$ together with appropriate tails $T_i$ of $R^r_i$ for each $i \in [r]$. Suppose firstly that $\sigma^{-1}(n) = r+1 \in \mathbb{N}$.

Let $X = V(H_{n-1})$. There is a linkage from $(T_i \colon i \in [r] )$ to $(R^{r+1}_1,R^{r+1}_2, \ldots , R^{r+1}_r)$ after $X$ by Lemma~\ref{l:weaklink}, and, after relabelling, we may assume this linkage induces the identity on $[r]$. Let us suppose the linkage consists of paths $P_i$ from $x_i \in T_i$ to $y_i \in R^{r+1}_i$.

Since $X \cup \bigcup_i P_i \cup \bigcup_i T_ix_i$ is a finite set, there is some vertex $y_{r+1}$ on $R^{r+1}_{r+1}$ such that the tail $y_{r+1}R^{r+1}_{r+1}$ is disjoint from $X \cup \bigcup_i P_i \cup \bigcup_i T_ix_i$.

	To form $H_n$ we add the paths $T_ix_i \cup P_i$ to the branch set of each $i \leq r$ and set $y_{r+1}$ as the branch set for $r+1$. Then $H_n$ is an $IG_n$ extending $H_{n-1}$ and the tails $y_jR^{r+1}_j$ are as claimed.

Suppose then that $\sigma^{-1}(n) = \{u,v\} \in \mathbb{N}^{(2)}$ with $u,v \leq r$. We have tails $T_i$ of $R^r_i$ for each $i \in [r]$ which are disjoint from $H_{n-1}$ apart from their initial vertices. Let us take tails $T_j$ of $R^r_j$ for each $j > r$ which are also disjoint from $H_{n-1}$. Since $\RG(\mathcal{R}_r)$ is $r$-pebble-win, it follows that $\RG(T_i \colon i \in [m_r])$ is also $r$-pebble-win. Furthermore, since by Lemma~\ref{l:connraygraph} $\RG(T_i \colon i \in [m_r])$ is connected, there is some neighbour $w \in [m_r]$ of $u$ in $\RG(T_i \colon i \in [m_r])$.

Let us first assume that $w\notin [r]$. Since $\RG(T_i \colon i \in [m_r])$ is $r$-pebble-win, the game state $(1,2,\ldots,v-1,w,v+1,\ldots,r)$ is an achievable game state in the $(1,2, \ldots ,r)$- pebble-pushing game and hence by Lemma~\ref{l:transitionpebble} the function $\varphi_1$ given by $\varphi_1(i) = i$ for all $i \in [r] \setminus \{v\}$ and $\varphi_1(v)=w$ is a transition function from $(T_i \colon i \in [r])$ to $(T_i \colon i \in [m_r])$.

Let us take a linkage from $(T_i \colon i \in [r])$ to $(T_i \colon i \in [m_r])$ inducing $\varphi_1$ which is after $V(H_{n-1})$. Let us suppose the linkage consists of paths $P_i$ from $x_i \in T_i$ to $y_i \in T_i$ for $i \neq v$ and $P_v$ from $x_v\in T_v$ to $y_v \in T_w$. Let 
\[
X = V(H_{n-1}) \cup \bigcup_{i \in [r]} P_i \cup \bigcup_{i \in [r]} T_ix_i
\]

Since $u$ is adjacent to $w$ in $\RG(T_i \colon i \in [m_r])$ there is a path $\hat{P}$ between $T(T_u,X)$ and $T(T_w,X)$ which is disjoint from $X$ and from all other $T_i$, say $\hat{P}$ is from $\hat{x} \in T_u$ to $\hat{y} \in T_w$.

Finally, since $\RG(T_i \colon i \in [m_r])$ is $r$-pebble-win, the game state $(1,2, \ldots ,r)$ is an achievable game state in the $(1,2,\ldots,v-1,w,v+1,\ldots,r)$-pebble-pushing game and hence by Lemma~\ref{l:transitionpebble} the function $\varphi_2$ given by $\varphi_2(i) = i$ for all $i \in [r] \setminus \{v\}$ and $\varphi_2(w)=v$ is a transition function from $(T_i \colon i \in [r] \setminus \{v\} \cup \{w\})$ to $(T_i \colon i \in [m_r])$.

Let us take a further linkage from $(T_i \colon i \in [r] \setminus \{v\} \cup \{w\})$ to $(T_i \colon i \in [m_r])$ inducing $\varphi_2$ which is after $X \cup \hat{P} \cup T_u \hat{x} \cup y_vT_w\hat{y}$. Let us suppose the linkage consists of paths $P'_i$ from $x'_i \in T_i$ to $y'_i \in T_i$ for $i \in [r] \setminus \{v\}$ and $P'_v$ from $x'_v\in T_w$ to $y'_v \in T_v$.

In the case that $w\in [r]$, $w<v$, say, the game state 
$$(1,2,\ldots,w~-~1,v,w~+~1,\ldots, v~-~1,w,v~+~1,\ldots r)$$
is an achievable game state in the $(1,2, \ldots ,r)$-pebble pushing-game and we get, by a similar argument, all $P_i,x_i,y_i,P_i',x_i',y_i'$ and $\hat{P}$.

We build $H_n$ from $H_{n-1}$ by adjoining the following paths:

\begin{itemize}
\item for each $i \neq v$ we add the path $T_ix_i P_i y_i T_i x'_i P'_i y'_i$ to $H_{n-1}$, adding the vertices to the branch set of $i$;
\item we add $\hat{P}$ to $H_{n-1}$, adding the vertices of $V(\hat{P}) \setminus \lbrace \hat{y} \rbrace$ to the branch set of $u$;
\item we add the path $T_v x_v P_v y_v T_w x'_v P'_v y'_v$ to $H_{n-1}$, adding the vertices to the branch set of $v$.
\end{itemize}

We note that, since $\hat{y} \in y_v T_w x'_v$ the branch sets for $u$ and $v$ are now adjacent. Hence $H_n$ is an $IG_n$ extending $H_{n-1}$. Finally the rays $y'_i T_i$ for $i \in [r]$ are appropriate tails of the used rays of $\mathcal{R}_r$.
\end{proof}

As every countable graph is a subgraph of~$K_{\aleph_0}$, a graph with a pebbly end contains every countable graph as a minor. 
Thus, as~${\aleph_0 G}$ is countable, if~$G$ is countable, we obtain the following corollary: 

\begin{corollary}
    \label{c:pebblyubiq}
    Let~$\Gamma$ be a graph with a pebbly end~$\omega$ and let~$G$ be a countable graph. 
    Then~${\aleph_0 G \preceq \Gamma}$.
\end{corollary}

So, at least when considering the question of $\preceq$-ubiquity for countable graphs, 
Corollary~\ref{c:pebblyubiq} allows one to restrict one's attention to host graphs~$\Gamma$ in which each end is non-pebbly. 
For this reason it will be useful to understand the structure of such ends. 

On immediate observation we can make is the following corollary of Lemma~\ref{l:kpebbstructure}.

\begin{corollary}
    \label{c:pebblyendstructure}
    Let~$\omega$ be an end of a graph~$\Gamma$ which is not $k$-pebble-win for some positive integer~$k$ and let~${\mathcal{R} = (R_i \colon i \in [m])}$ be a family of~${m \geq k+2}$ disjoint rays in~$\omega$. 
    Then there is a bare path~${P = p_1p_2\ldots p_n}$ in~${\RG(R_i \colon i \in [m])}$ such that~${|[m] \setminus V(P)| \leq k}$. 
    Furthermore, either each edge in~$P$ is a bridge in~${\RG(R_i \colon i \in [m])}$, or~${\RG(R_i \colon i \in [m])}$ is a cycle.
\end{corollary}

So, if~$\omega$ is not pebbly, then the ray graph of every family of $\omega$-rays is either close in structure to a path, or close in structure to a cycle. 
In fact, this dichotomy is not just true for each ray graph individually, but rather uniformly for each ray graph in the end. 
That is, we will show that either every ray graph of a family of $\omega$-rays will be close in structure to a path, or every ray graph will be close in structure to a cycle. 
Furthermore, the structure of this end will restrict the possible transition functions between families of $\omega$-rays.

As motivating examples consider the half-grid~${\mathbb{N} \square \mathbb{Z}}$ and the full-grid~${\mathbb{Z} \square \mathbb{Z}}$. 
Both graphs have a unique end~${\omega_h / \omega_f}$ and it is easy to show that the ray graph of every family of $\omega_h$-rays is a path, and the ray graph of every family of $\omega_f$-rays is a cycle (and so in particular $\mathbb{N} \square \mathbb{Z}$ is not $2$-pebble-win and $\mathbb{Z} \square \mathbb{Z}$ is not $3$-pebble-win). 

There is a natural way to order any family of disjoint $\omega_h$-rays, if you imagine them drawn on a page their tails will appear in some order from left to right. 
Then, it can be shown that any transition function between two large enough families of $\omega_h$-rays must preserve this ordering.

Similarly, there is a natural way to cyclically order any family of disjoint $\omega_f$-rays. 
As before, it can be shown that any transition function between two large enough families of $\omega_f$-rays must preserve this ordering.

The aim of the next few sections is to demonstrate that the above dichotomy holds for all non-pebbly ends: that either every ray-graph is close in structure to a path or close in structure to a cycle, and furthermore that in each of these cases the possible transition functions between families of rays are restricted in a similar fashion as those of the half-grid or full-grid, in which case we will say the end is \emph{half-grid-like} or \emph{grid-like} respectively. These results, whilst not used in this paper, will be a vital part of the proof in~\cite{BEEGHPTIII}.

We note that, in principle, this trichotomy that an end of a graph is either pebbly, grid-like or half-grid-like, and the information that this implies about its finite rays graphs and the transitions between them, could in principle be derived from earlier work of Diestel and Thomas~\cite{DT99}, who gave a structural characterisation of graphs without a $K_{\aleph_0}$-minor. 
However, to introduce their result and derive what we needed from it would have been at least as hard as our work in Section~\ref{s:nonpebbly}, if not more complicated, and so we have opted for a straightforward and self-contained presentation.

\section{The structure of non-pebbly ends}
\label{s:nonpebbly}
\subsection{Polypods}
\label{s:polypod}

It will be useful for our analysis of the structure of non-pebbly ends to consider the possible families of disjoint rays in the end with a fixed set of start vertices, and the relative structure of these rays. 

\begin{definition}
    \label{def_polypod}
    Given an end~$\epsilon$ of a graph~$\Gamma$, a {\em polypod (for~$\epsilon$ in~$\Gamma$)} is a pair~${(X, Y)}$ of disjoint finite sets of vertices of~$\Gamma$ 
    such that there is at least one family ${(R_y\colon y \in Y)}$ of disjoint~$\epsilon$-rays, where~$R_y$ begins at~$y$ and all the~$R_y$ are disjoint from~$X$. 
    Such a family ${(R_y\colon y \in Y)}$ is called a {\em family of tendrils} for~${(X,Y)}$. 
    The {\em order} of the polypod is~${|Y|}$. 
    The {\em connection graph}~${K_{X,Y}}$ of a polypod~${(X,Y)}$ is a graph with vertex set~$Y$. 
    It has an edge between vertices~$v$ and~$w$ if and only if there is a family ${(R_y\colon y \in Y)}$ of tendrils for~${(X,Y)}$ such that there is an $R_v$--$R_w$-path in~$\Gamma$ disjoint from~$X$ and from every other~$R_y$. 
\end{definition}

Note that the ray graph of any family of tendrils for a polypod must be a subgraph of the connection graph of that polypod. 

\begin{definition}
    We say that a polypod~${(X,Y)}$ for~$\epsilon$ in~$\Gamma$ is {\em tight} if its connection graph is minimal amongst connection graphs of polypods for~$\epsilon$ in~$\Gamma$ with respect to the spanning isomorphic subgraph relation, i.e.~for no other polypod~${(X',Y')}$ for~$\epsilon$ in~$\Gamma$ of order ${|Y'| = |Y|}$ is the graph~${K_{X',Y'}}$ isomorphic to a proper subgraph of~$K_{X,Y}$. 
    (Let us write ${H \subsetsim G}$ if~$H$ is isomorphic to a subgraph of~$G$.)
    We say that a polypod {\em attains} its connection graph if there is some family of tendrils for that polypod whose ray graph is equal to the connection graph.
\end{definition}

\begin{lemma}
    \label{l:growpolypod}
    Let~${(X,Y)}$ be a tight polypod, ${(R_y \colon y\in Y)}$ a family of tendrils and for every~${y \in Y}$ let~$v_y$ be a vertex on~$R_y$.
    Let~$X'$ be a finite vertex set disjoint from all~${v_y R_y}$ and including~$X$ as well as each of the initial segments~${R_y \mathring{v}_y}$.
    Let ${Y' = \{v_y \colon y\in Y\}}$. 
    Then~${(X', Y')}$ is a tight polypod with the same connection graph as~${(X,Y)}$.
\end{lemma}

\begin{proof}
    The family ${(v_yR_y \colon y\in Y)}$ witnesses that~${(X', Y')}$ is a polypod. 
    Moreover every family of tendrils for~${(X', Y')}$ can be extended by the paths~${R_y v_y}$ to obtain a family of tendrils for~${(X,Y)}$.
    Hence if there is an edge~${v_yv_z}$ in~$K_{X' Y'}$ then there must also be the edge~$yz$ in~$K_{X,Y}$. 
    Thus~${K_{X',Y'} \subsetsim K_{X,Y}}$.
    But since~${(X,Y)}$ is tight we must have equality. 
    Therefore~${(X',Y')}$ is tight as well.
\end{proof}

\begin{lemma}
    \label{l:polypod}
    Any tight polypod~${(X,Y)}$ attains its connection graph.
\end{lemma}

\begin{proof}
    We must construct a family of tendrils for~${(X,Y)}$ whose ray graph is~$K_{X,Y}$. 
    We will recursively build larger and larger initial segments of the rays, together with disjoint paths between them. 
    
    Precisely this means that, after partitioning~$\Nbb$ into infinite sets~$A_e$,  one for each edge~$e$ of~$K_{X,Y}$, we will construct, for each~${n \in \Nbb}$, a family ${(P^n_y \colon y \in Y)}$ of disjoint paths, and also paths~$Q_n$ such that for some arbitrary fixed ray $R  \in \epsilon$:
    
    \begin{itemize}
        \item \label{ten:start} Each~$P^n_y$ starts at~$y$. We write $y_n$ for the last vertex of $P^n_y$.
        \item \label{ten:pathlength} Each~$P^n_y$ has length at least~$n$ and there are at least $n$ disjoint paths from $P^n_y$ to~$R$. 
        \item \label{ten:extend} For~${m \leq n}$, the path~$P^n_y$ extends~$P^m_y$. 
        \item \label{ten:Qjoin} If~${n \in A_{vw}}$ for~${vw \in E(K_{X, Y})}$, then~$Q_n$ is a path from~$P^n_v$ to~$P^n_w$. 
        \item \label{ten:QPdisjoint} If~${n \in A_{vw}}$ for~${vw \in E(K_{X, Y})}$, then~$Q_n$ meets no~$P^m_y$ with~${y \in Y \setminus \{v, w\}}$ for \linebreak any~${m \in \Nbb}$. 
        \item \label{ten:Qdisjoint} All the~$Q_n$ are pairwise disjoint.
        \item \label{ten:Xdisjoint} All the~$P^n_y$ and all the~$Q_n$ are disjoint from~$X$. 
        \item \label{ten:extendable} For any~${n \in \Nbb}$ there is a family ${(R^n_y \colon y \in Y)}$ of tendrils for~${(X,Y)}$ such that each~$P^n_y$ is an initial segment of the corresponding~$R^n_y$, and the~$R^n_y$  meet the~$Q_m$ with~${m \leq n}$ in  $P^n_y \mathring{y}_n$.
    \end{itemize}
    
    Once the construction is complete, we obtain a family of tendrils by letting each~$R_y$ be the union of all the~$P^n_y$ -- indeed, $R_y$ clearly is an $\epsilon$-ray since there are arbitrarily many disjoint paths from $R_y$ to $R$.
    Furthermore, for any edge~$e$ of~$K_{X,Y}$ the family ${(Q_n \colon n \in A_e)}$ will witness that~$e$ is in the ray graph of this family. 
    So that ray graph will be all of~$K_{X,Y}$, as required. 
    
    So it remains to show how to carry out this recursive construction. 
    Let~$vw$ be the edge of~$K_{X,Y}$ with~${1 \in A_{vw}}$. 
    By the definition of the connection graph there is a family ${(R^1_y \colon y \in Y)}$ of tendrils for~${(X,Y)}$ such that there is a path~$Q_1$ from~$R^1_v$ to~$R^1_w$, disjoint from all other~$R^1_y$ and from~$X$. 
    
    For each~${y \in Y}$ let~$P^1_y$ be an initial segment of~$R^1_y$ with end vertex $y_1$ of length at least~$1$ such that $Q_1 \cap R^1_y  \subseteq P^1_y \mathring{y}_1$, and such that there is a path from $P^1_z$ to $R$ -- which is possible, since both $R$ and $R^1_y$ are $\epsilon$-rays.
    
    This choice of the~$P^1_y$ and of~$Q_1$ clearly satisfies the conditions above. 
    
    Suppose that we have constructed suitable~$P^m_y$ and~$Q_m$ for all~${m \leq n}$. 
    For each~${y \in Y}$, let~$y_n$ be the endvertex of~$P^n_y$. 
    Let~$Y_n$ be~${\{y_n \colon y \in Y\}}$ and 
    \[ 
        Z_n = X \cup \bigcup_{m \leq n} \bigcup_{y\in Y} \left(V(P^m_y) \cup V(Q_m)\right). 
    \] 
    Let~$X_n$ be~${Z_n \setminus Y_n}$, and note that every~${V(Q_m) \subseteq X_n}$ for every~${m \leq n}$. 
    Then by Lemma~\ref{l:growpolypod} ${(X_n, Y_n)}$ is a tight polypod with the same connection graph as~${(X,Y)}$. 
    
    In particular, letting~$vw$ be the edge of~$K_{X,Y}$ with~${n+1 \in A_{vw}}$, we have that~${v_nw_n}$ is an edge of~$K_{X_n,Y_n}$. 
    So there is a family ${(S^{n+1}_{y_n} \colon y_n \in Y_n)}$ of tendrils for~${(X_n, Y_n)}$ together with a path~$Q_{n+1}$ from~$S^{n+1}_{v_n}$ to~$S^{n+1}_{w_n}$ disjoint from all other~$S^{n+1}_{y_n}$ and from~$X_n$. 
    Now for any~${y \in Y}$ we let~$R^{n+1}_y$ be the ray~${yP^n_yy_nS^{n+1}_{y_n}}$.
    Let~$P^{n+1}_y  = R^{n+1}_y y_{n+1}$ be an initial segment of~$R^{n+1}_y$  of length at least~${n+1}$ and long enough to include~$P^n_y$, and such that ${Q_{n+1} \cap R^{n+1}_y} \subseteq P^{n+1}_y \mathring{y}_{n+1}$, and such that there are at least $n+1$-disjoint paths between $P^{n+1}_y$ and $R$ - which is possible since both $R$ and $R^{n+1}_y$ are $\epsilon$-rays. This completes the recursion step, and so the construction is complete. 
\end{proof}

\begin{lemma}
    \label{l:raysubgraph}
    Let~${(X,Y)}$ be a polypod of order~$n$ for~$\epsilon$ in~$\Gamma$ with connection graph~$K_{X,Y}$, $(S_y \colon y \in Y)$ be a family of tendrils for $(X,Y)$, and $(R_i \colon i \in I)$ be a set of disjoint $\epsilon$-rays. Then for any transition function $\sigma$ from $\mathcal{S}$ to $\mathcal{R}$ and every pair $y,y' \in Y$ such that there is a path from $\sigma(y)$ to $\sigma(y')$ otherwise avoiding $\sigma(Y)$ in $E(\RG(R_i \colon i \in I))$, the edge  $yy'$ is in $E(K_{X,Y})$.
\end{lemma}

\begin{proof}
    Since $\sigma$ is a transition function there exists a linkage from $\mathcal{S}$ to $\mathcal{R}$ after $X$ which induces $\sigma$. This linkage gives us a family of tendrils $(S'_y \colon y \in Y)$ for~${(X,Y)}$ such that $S'_y$ is a tail of $R_{\sigma(y)}$ for each $y \in Y$. Then, by Lemmas~\ref{l:connraygraph} and~\ref{l:raygraph_subfamily}, if $y,y' \in Y$ are such that there is a path from $\sigma(y)$ to $\sigma(y')$ otherwise avoiding $\sigma(Y)$ in $E(\RG(R_i \colon i \in I))$, then $S'_y$ and $S'_{y'}$ are adjacent in $\RG(S'_y \colon y \in Y)$, and so $y$ and $y'$ are adjacent in $K_{X,Y}$.
\end{proof}

\begin{cordef}
    Any two polypods for~$\epsilon$ in~$\Gamma$ of the same order which attain their connection graphs have isomorphic connection graphs. 
    
    We will refer to the graph arising in this way for polypods of order~$n$ for~$\epsilon$ in~$\Gamma$ as the~$n$\textsuperscript{th} {\em shape graph} of the end~$\epsilon$.
    \nobreak\hfill\qedsymbol
\end{cordef}

\subsection{Frames}
Given a family of tendrils $(R_y \colon y \in Y)$ for a polypod $(X,Y)$ there may be different families of tendrils $(R'_y \colon y \in Y)$ for $(X,Y)$ such that each $R_y$ shares a tail with some $R'_{\pi(y)}$. In order to understand the possible transition functions between different families of rays in $\epsilon$ it will be useful to understand the possible functions $\pi$ that arise in this fashion.

To do so will we consider \emph{frames}, finite subgraphs $L$ which contain a path family between two sets of vertices $\alpha(Y)$ and $\beta(Y)$. For appropriate choices of $\alpha(Y)$ and $\beta(Y)$ these will be the subgraphs arising from a linkage from the family of tendrils $(R_y \colon y \in Y)$ to itself after $X$, each of which gives rise to a family $(R'_y \colon y \in Y)$ as above.

Some frames will contain multiple such path families, linking $\alpha(Y)$ to $\beta(Y)$ in different ways. 
For appropriately chosen frames the possible ways we can link $\alpha(Y)$ to $\beta(Y)$ will be restricted by the structure of $K_{X,Y}$, which will allow us relate this to the possible transition functions from $(R_y \colon y \in Y)$ to itself, and from there to the possible transition functions between different families of rays.

\begin{definition}
    Let~$Y$ be a finite set. 
    A {\em $Y$-frame}~${(L,\alpha,\beta)}$ consists of a finite graph~$L$ together with two injections~$\alpha$ and~$\beta$ from~$Y$ to~${V(L)}$. 
    The set~${A = \alpha(Y)}$ is called the {\em source set} and the set~${B = \beta(Y)}$ is called the {\em target set}. 
    A {\em weave} of the $Y$-frame is a family ${\Qcal = (Q_y \colon y \in Y)}$ of disjoint paths in~$L$ from~$A$ to~$B$, where the initial vertex of~$Q_y$ is~$\alpha(y)$ for each~${y \in Y}$. 
    The {\em weave pattern}~$\pi_{\Qcal}$ of~$\Qcal$ is the bijection from~$Y$ to itself sending~$y$ to the inverse image under~$\beta$ of the endvertex of~$Q_y$. 
    In other words, $\pi_{\Qcal}$ is the function so that every~$Q_y$ is an $\alpha(y)$--$\beta(\pi_{\Qcal}(y))$ path. 
    The {\em weave graph}~$K_{\Qcal}$ of~$\Qcal$ has vertex set~$Y$ and an edge joining distinct vertices~$u$ and~$v$ of~$Y$ precisely when there is a path from~$Q_u$ to~$Q_v$ in~$L$ disjoint from all other~$Q_y$. 
    For a graph~$K$ with vertex set~$Y$, we say that the $Y$-frame is {\em $K$-spartan} if all its weave graphs are subgraphs of~$K$ and all its weave patterns are automorphisms of~$K$.
\end{definition} 

Connection graphs of polypods and weave graphs of frames are closely connected.

\begin{lemma}
    \label{l:getspartan}
    Let~${(X,Y)}$ be a polypod for~$\epsilon$ in~$\Gamma$ attaining its connection graph~$K_{X,Y}$ and let ${\mathcal{R}=(R_y \colon y \in Y)}$ be a family of tendrils for~${(X,Y)}$. 
    Let~$L$ be any finite subgraph of~$\Gamma$ disjoint from~$X$ but meeting all the~$R_y$. 
    For each~${y \in Y}$ let~$\alpha(y)$ be the first vertex of~$R_y$ in~$L$ and~$\beta(y)$ the last vertex of~$R_y$ in~$L$. 
    Then the $Y$-frame ${(L, \alpha, \beta)}$ is $K_{X,Y}$-spartan.
\end{lemma}

\begin{proof}
    Since there is some family of tendrils ${(S_y \colon y\in Y)}$ attaining~$K_{X,Y}$ and there is, by Lemma~\ref{l:weaklink}, a linkage from ${(R_y \colon y \in Y)}$ to ${(S_y\colon y\in Y)}$ after~$X$ and~${V(L)}$, 
    we may assume without loss of generality that ${\RG(R_y \colon y\in Y)}$ is isomorphic to~$K_{X,Y}$. 
    
    For a given weave ${\Qcal = (Q_y \colon y \in Y)}$, applying the definition of the connection graph to the rays ${R'_y=R_y\alpha(y)Q_y\beta(\pi_{\Qcal}(y))R_{\pi_{\Qcal}(y)}}$ shows that~$K_{\Qcal}$ is a subgraph of~$K_{X,Y}$. Furthermore, since ${\RG(R_y \colon y\in Y)}$ is isomorphic to $K_{X,Y}$, for any $uv \in E(K_{X,Y})$ there is a path from $R_u$ to $R_v$ which is disjoint from $R_u\alpha(u) \cup R_v\alpha(v) \cup L \cup X$ and which doesn't meet any other $R_y$, and so joins $R'_{\pi^{-1}(u)}$ to $R'_{\pi^{-1}(v)}$. So, the family of tendrils $(R'_y \colon y \in Y)$ witness that $\pi^{-1}(u)\pi^{-1}(v) \in E(K_{X,Y})$, and so $\pi_{\Qcal}$ is an automorphism of~$K_{X,Y}$.
\end{proof}

\begin{corollary}
    \label{c:transitions}
    Let~${(X,Y)}$ be a polypod for~$\epsilon$ in~$\Gamma$ attaining its connection graph~$K_{X,Y}$ and let ${\Rcal = (R_y \colon y \in Y)}$ be a family of tendrils for~${(X,Y)}$. 
    Then for any transition function~$\sigma$ from~$\Rcal$ to itself there is a $K_{X,Y}$-spartan $Y$-frame for which both~$\sigma$ and the identity are weave patterns. 
\end{corollary}

\begin{proof}
    Let~${(P_y \colon y\in Y)}$ be a linkage from~$\Rcal$ to itself after~$X$ inducing~$\sigma$, and let~$L$ be a finite subgraph graph of~$\Gamma$ containing~$\bigcup_{y \in Y} P_y$ as well as a finite segment of each~$R_y$, such that each~$P_y$ is a path between two such segments. 
    Then the~$Y$-frame on~$L$ which exists by Lemma~\ref{l:getspartan} has the desired properties. 
\end{proof}

\begin{lemma}
    \label{l:raygraphs}
    Let~${(X,Y)}$ be a polypod for~$\epsilon$ in~$\Gamma$ attaining its connection graph~$K_{X,Y}$ and let ${\mathcal{R}=(R_y \colon y \in Y)}$ be a family of tendrils for~${(X,Y)}$. 
    Then there is a $K_{X,Y}$-spartan $Y$-frame for which both~$K_{X,Y}$ and~${\RG(R_y \colon y\in Y)}$ are weave graphs.
\end{lemma}

\begin{proof}
    By adding finitely many vertices to~$X$ if necessary, we may obtain a superset~$X'$ of~$X$ such that for any two of the~$R_y$, if there is any path between them disjoint from all the other rays and~$X'$, then there are infinitely many disjoint such paths. 
    Let ${(S_y \colon y \in Y)}$ be any family of tendrils for~${(X,Y)}$ with connection graph~$K_{X,Y}$.
    
    For each edge~${e = uv}$ of~$\RG(\mathcal{R})$ let~$P_e$ be a path from~$R_u$ to~$R_v$ disjoint from all the other~$R_y$ and from~$X'$. 
    Similarly for each edge~${f = uv}$ of~$K_{X,Y}$ let~$Q_f$ be a path from~$S_u$ to~$S_v$ disjoint from all the other~$S_y$ and from~$X'$. 
    Let~${(P'_y \colon y \in Y)}$ be a linkage from the~$S_y$ to the~$R_y$ after 
    \[
        X' \cup \bigcup_{e\in E(\RG(\mathcal{R}))}P_e \cup \bigcup_{f \in E(K_{X,Y})}Q_f.
    \]
    Let the initial vertex of~$P'_y$ be~$\gamma(y)$ and the end vertex be~$\beta(y)$. 
    Let~$\pi$ be the permutation of~$Y$ by setting~$\pi(y)$ to be the element of~$Y$ with~$\beta(y)$ on~$R_{\pi(y)}$. 
    Let $L$ be the graph given by the union of all paths of the form $S_y\gamma(y)$ and $R_{\pi(y)}\beta(y)$ together with $P'_y$, $P_e$ and $Q_e$.
    
    Letting~$\alpha$ be the identity function on~$Y$, it follows from Lemma~\ref{l:getspartan} that ${(L, \alpha, \beta)}$ is a $K_{X,Y}$-spartan $Y$-frame. 
    The paths~$Q_f$ witness that the weave graph for the paths~${S_y\gamma(y)P'_y}$ includes~$K_{X,Y}$ and so, by $K_{X,Y}$-spartanness, must be equal to~$K_{X,Y}$. 
    The paths~$P_e$ witness that the weave graph for the paths~${R_y\beta(y)}$ includes the ray graph~${\RG(\mathcal{R})}$. 
    However conversely, since~$V(L)$ is disjoint from~$X'$, if two of the~$R_y$ are joined in~$L$ by a path disjoint from the other rays in~$\mathcal{R}$ then they are joined by infinitely many, and hence adjacent in~${\RG(\mathcal{R})}$. 
    It follows that the weave graph is equal to~${\RG(\mathcal{R})}$.
\end{proof}

Hence to understand ray graphs and the transition functions between them it is useful to understand the possible weave graphs and weave patterns of spartan frames. Their structure can be captured in terms of automorphisms and cycles.

\begin{definition}
    Let~$K$ be a finite graph. 
    An automorphism~$\sigma$ of~$K$ is called {\em local} if it is a cycle ${(z_1 \ldots z_t)}$ where, for any~${i \leq t}$, there is an edge from~$z_i$ to~${\sigma(z_i)}$ in~$K$. 
    If~${t \geq 3}$ this means that~${z_1 \ldots z_tz_1}$ is a cycle of~$K$, and we call such cycles {\em turnable}. 
    If~${t = 2}$ then we call the edge~${z_1z_2}$ of~$K$ {\em flippable}. 
    We say that an automorphism of~$K$ is {\em locally generated} if it is a product of local automorphisms. 
\end{definition}

\begin{remark}
    \label{r:cutvertices}
    A cycle~$C$ in~$K$ is turnable if and only if all its vertices have the same neighbourhood in~${K - C}$, and whenever a chord of length~${\ell \in \Nbb}$, i.e.~a chord whose endvertices have distance~$\ell$ on~$C$, is present in~${K[C]}$, then all chords of length~$\ell$ are present. 
    Similarly an edge~$e$ of~$K$ is flippable if and only if its two endvertices have the same neighbourhood in~${K-e}$. 
    Thus, if~$K$ is connected and contains at least three vertices, no vertex of degree one or cutvertex of~$K$ can lie on a turnable cycle or a flippable edge. 
    So vertices of degree one and cutvertices in such graphs are preserved by locally generated automorphisms.
\end{remark}

\begin{lemma}
    \label{l:controlframe}
    Let $\mathcal{L}=(L, \alpha, \beta)$ be a $K$-spartan $Y$-frame which is $K$-spartan. Then for any two of its weave patterns~$\pi$ and~$\pi'$ the automorphism~${\pi^{-1}\cdot\pi'}$ of~$K$ is locally generated. Furthermore, if $K$ is a weave graph for $\mathcal{L}$ then each weave graph for $\mathcal{L}$ contains a turnable cycle or a flippable edge of~$K$.
\end{lemma}
\begin{proof}
    Let us suppose, for a contradiction, that the conclusion does not hold and let ${\mathcal{L} = (L, \alpha, \beta)}$ be a counterexample in which~${|E(L)|}$ is minimal. 
    Let ${\Pcal = (P_y \colon y \in Y)}$ and ${\Qcal = (Q_y \colon y \in Y)}$ be weaves for $\mathcal{L}$ such that either $\pi_\Pcal \neq \pi_\Qcal$ and~${\pi_\Pcal^{-1}\cdot\pi_\Qcal}$ is not locally generated, or $K_\Pcal = K$ and $K_\Qcal$ does not contain a turnable cycle or a flippable edge of~$K$.
    
    Each edge of~$L$ is in one path of~$\Pcal$ or~$\Qcal$ since otherwise we could simply delete it. 
    Similarly no edge appears in both~$\Pcal$ and~$\Qcal$ since otherwise we could simply contract it. 
    No vertex appears on just one of~$P_y$ or~$Q_y$ since otherwise we could contract one of the two incident edges. 
    Vertices of~$L$ appearing in neither~${\bigcup \Pcal}$ nor~${\bigcup \Qcal}$ are isolated and so may be ignored. 
    Thus we may suppose that each edge of~$L$ appears in precisely one of~$\Pcal$ or~$\Qcal$, and that each vertex of~$L$ appears in both. 
    
    Let~$Z$ be the set of those~${y \in Y}$ for which~${\alpha(y) \neq \beta(y)}$. 
    For any~${z \in Z}$ let~$\gamma(z)$ be the second vertex of~$P_z$, i.e.~the neighbour of~$\alpha(z)$ on~$P_z$, and let~${f(z) \in Y}$ be chosen such that~$\gamma(z)$ lies on~$Q_{f(z)}$. 
    Then since~${\gamma(z) \neq \alpha(f(z))}$ we have~${f(z) \in Z}$ for all~${z \in Z}$. 
    Furthermore, $Z$ is nonempty as~$\Pcal$ and~$\Qcal$ are distinct. 
    Let~$z$ be any element of~$Z$. 
    Then since~$Z$ is finite there must be~${i < j}$ with ${f^i(z) = f^j(z)}$, which means that~${f^i(z) = f^{j-i}(f^i(z))}$. 
    Let~${t>0}$ be minimal such that there is some~${z_1 \in Z}$ with~${z_1 = f^t(z_1)}$. 
    
    If~${t = 1}$ then we may delete the edge~${\alpha(z_1)\gamma(z_1)}$ and replace~$P_{z_1}$ with~${\alpha(z_1)Q_{z_1}\gamma(z_1)P_{z_1}}$. 
    This preserves all of~$\pi_{\Pcal}$, $\pi_{\Qcal}$ and~$K_{\Qcal}$, and can only make~$K_{\Pcal}$ bigger, contradicting the minimality of our counterexample. 
    So we must have~${t \geq 2}$.
    
    For each~${i \leq t}$ let~$z_i$ be~${f^{i-1}(z_1)}$ and let~$\sigma$ be the bijection~${(z_1z_2 \ldots z_t)}$ on~$Y$. 
    Let~$L'$ be the graph obtained from~$L$ by deleting all vertices of the form~$\alpha(z_i)$. 
    Let~$\alpha'$ be the injection from~$Y$ to~$V(L')$ sending~$z_i$ to~$\gamma(z_i)$ for~${i \leq t}$ 
    and sending any other~${y \in Y}$ to~$\alpha(y)$. 
    Then ${(L', \alpha', \beta)}$ is a $Y$-frame. 
    For any weave ${(\hat P_y \colon y \in Y)}$ in this $Y$-frame, ${(P''_y : y \in Y)}$ where ${P''_{z_i} = \alpha(z_i) \gamma(z_i) \hat{P}_{z_i}}$ for every~${i\leq t}$ and~${P''_y = \alpha(y)}$ for every~${y \in Y \setminus \{z_1,\ldots,z_t\}}$
    is a weave in ${(L, \alpha, \beta)}$ with the same weave pattern and whose weave graph includes that of~${(\hat P_y \colon y \in Y)}$. 
    Thus~${(L', \alpha', \beta)}$ is $K$-spartan. 
    
    Let~$P'_y$ be~${\alpha'(y)P_y}$ and~$Q'_{y}$ be~${\alpha'(y)Q_{\sigma(y)}}$ for each~${y \in Y}$. 
    Now set ${\mathcal{P'} = (P'_y \colon y \in Y)}$ and~${\mathcal{Q'} = (Q'_y \colon y \in Y)}$. 
    Then we have ${\pi_{\Qcal'} = \pi_{\Qcal}\cdot \sigma}$ and so ${\sigma = \pi_{\Qcal}^{-1} \cdot \pi_{\Qcal'}}$ is an automorphism of~$K$ since~$\pi_Q$ is an automorphism of~$K$ by the $K$-spartanness. 
    For any~${i \leq t}$ the edge~${\alpha(z_i)\gamma(z_i)}$ witnesses that~${z_i \sigma(z_i)}$ is an edge of~$K_{\Qcal}$, and hence, since~$\mathcal{L}$ is~$K$-spartan, also an edge of~$K$, and so~$\sigma$ is a local automorphism of $K$. It follows that~$K_{\Qcal}$ includes a turnable cycle or a flippable edge. 
    Finally, by the minimality of~${|E(L)|}$ we know that~${\pi_{\Pcal'}^{-1}\cdot\pi_{\Qcal'}}$ is locally generated and hence so is ${\pi_{\Pcal}^{-1}\cdot\pi_{\Qcal} = \pi_{\Pcal'}^{-1}\cdot\pi_{\Qcal'}\cdot\sigma^{-1}}$. 
    This is the desired contradiction.
\end{proof}

Finally, the following two lemmas are the main conclusions of this section:

\begin{lemma}
    \label{l:raynopaths}
    Let~${(X,Y)}$ be a polypod attaining its connection graph~$K_{X,Y}$ such that~$K_{X,Y}$ is a cycle of length at least~$4$. 
    Then for any family of tendrils~$\Rcal$ for this polypod the ray graph is~$K_{X,Y}$. 
    Furthermore, any transition function from~$\Rcal$ to itself preserves each of the cyclic orientations of~$K_{X,Y}$.
\end{lemma}

\begin{proof}
    By Lemma~\ref{l:raygraphs} there is some $K_{X,Y}$-spartan $Y$-frame for which both~$K_{X,Y}$ and the ray graph~${\RG(\mathcal{R})}$ are weave graphs. 
    Since~$K_{X,Y}$ is a cycle of length at least~$4$ and hence has no flippable edges, the ray graph must include a cycle by Lemma~\ref{l:controlframe} and so since it is a subgraph of~$K_{X,Y}$ it must be the whole of~$K_{X,Y}$. 
    Similarly Lemma~\ref{l:controlframe} together with Corollary~\ref{c:transitions} shows that all transition functions must be locally generated and so must preserve the orientation.
\end{proof}

\begin{lemma}
    \label{l:barepath}
    Let~${(X, Y)}$ be a polypod attaining its connection graph~$K_{X,Y}$ such that~$K_{X,Y}$ includes a bare path~$P$ whose edges are bridges. 
    Let~$\Rcal$ be a family of tendrils for~${(X, Y)}$ whose ray graph is~$K_{X,Y}$. 
    Then for any transition function~$\sigma$ from~$\Rcal$ to itself, the restriction of~$\sigma$ to~$P$ is the identity.
\end{lemma}

\begin{proof}
    By Lemmas~\ref{c:transitions} and~\ref{l:controlframe} any transition function must be a locally generated automorphism of~$K_{X,Y}$, and so by Remark~\ref{r:cutvertices} it cannot move the vertices of the bare path, which are vertices of degree one or cutvertices.
\end{proof}

\section{Grid-like and half-grid-like ends}
\label{s:gridandhalfgrid}

We are now in a position to analyse the different kinds of thick ends which can arise in a graph in terms of the possible ray graphs and the transition functions between them. 
The first kind of ends are the pebbly ends, in which, by Corollary~\ref{c:pebblyubiq}, for any~$n$ we can find a family of~$n$ disjoint rays whose ray graph is~$K_n$ and for which every function~${\sigma \colon [n] \rightarrow [n]}$ is a transition function.

So, in the following let us fix a graph~$\Gamma$ with a thick non-pebbly end~$\epsilon$ and a number~${N \in \Nbb}$, where~${N \geq 3}$, such that~$\epsilon$ is not $N$-pebble win. 
Under these circumstances we get nontrivial restrictions on the ray graphs and the transition functions between them. 
There are two essentially different cases, corresponding to the two cases in Corollary~\ref{c:pebblyendstructure}: 
The grid-like and the half-grid-like case.

\subsection{Grid-like ends}

The first case focuses on ends which behave like that of the infinite grid. In this case, all large enough ray graphs are cycles and all transition functions between them preserve the cyclic order. 

Formally, we say that the end~$\epsilon$ is {\em grid-like} if the ${(N + 2)}$\textsuperscript{nd} shape graph for~$\epsilon$ is a cycle. 
For the rest of this subsection we will assume that~$\epsilon$ is grid-like. 
Let us fix some polypod~${(X,Y)}$ of order~${N+2}$ attaining its connection graph. 
Let~${(S_y \colon y \in Y)}$ be a family of tendrils for~${(X,Y)}$ whose ray graph is the cycle~${C_{N+2} = K_{X,Y}}$.

\begin{lemma}
    \label{l:raygraphsarecycles}
    The ray graph~$K$ for any family~${(R_i \colon i \in I)}$ of $\epsilon$-rays in~$\Gamma$ with~${|I| \geq N+2}$ is a cycle. 
\end{lemma}

\begin{proof}
    By Corollary~\ref{c:pebblyendstructure}, $K$ is either a cycle or contains a bridge. 
    However, given any edge~${ij \in E(K)}$, let~${J \subseteq I}$ be such that~${i,j \in J}$ and~${|J|=N+2}$. 
    Let ${(T_y \colon y \in Y)}$ be a family of tendrils for ${(X,Y)}$ obtained by transitioning from~$(S_y \colon y \in Y)$ to~${(R_j \colon j \in J)}$ after~$X$ along some linkage. 
    By Lemma~\ref{l:connraygraph}, the ray graph $K_J$ of ${(R_i \colon i \in J)}$ is isomorphic to the ray graph of ${(T_y \colon y \in Y)}$, which is a cycle by Lemma~\ref{l:raynopaths}.

    Hence, $ij$ is not a bridge of~$K_J$, and it is easy to see that this implies that~$ij$ is not a bridge of~$K$. Hence, $K$ is a cycle. 
\end{proof}

Given a cycle~$C$ a \emph{cyclic orientation of~$C$} is an orientation of the graph~$C$ which does not have any sink. Note that any cycle has precisely two cyclic orientations. Given a cyclic orientation and three distinct vertices~${x, y, z}$ we say that they \emph{appear consecutively in the order~${(x, y, z)}$} if~$y$ lies on the unique directed path from~$x$ to~$z$. Given two cycles~${C, C'}$, each with a cyclic orientation, we say that an injection~${f \colon V(C) \to V(C')}$ \emph{preserves the cyclic orientation} if whenever three distinct vertices~${x,y}$ and~$z$ appear on~$C$ in the order~${(x,y,z)}$ then their images appear on~$C'$ in the order~${(f(x),f(y),f(z))}$.

We will now choose cyclic orientations of every large enough ray-graph such that the transition functions preserve the cyclic orders corresponding to those orientations. To that end, we fix a cyclic orientation of $K_{X,Y}$. We say that a cyclic orientation of the ray graph for a family $(R_i \colon i \in I)$ of at least $N+3$ disjoint $\epsilon$-rays is {\em correct} if there is a transition function $\sigma$ from the $S_y$ to the $R_i$ which preserves the cyclic orientation of $K_{X,Y}$.

\begin{lemma}
    \label{l:correctorientationcyclic}
    For any family ${(R_i \colon i \in I)}$ of at least~${N+3}$ disjoint $\epsilon$-rays there is precisely one correct cyclic orientation of its ray graph. 
\end{lemma}

\begin{proof}
    We first claim that there is at least one correct cyclic orientation. 
    By Lemma~\ref{l:weaklink}, there is a transition function~$\sigma$ from the~$S_y$ to some subset~$J$ of~$I$, 
    and we claim that there is some cyclic orientation of the ray graph~$K$ of~${(R_i \colon i \in I)}$ such that~$\sigma$ preserves the cyclic orientation of~$K_{X,Y}$.
    
    We first note that the ray graph~$K_J$ of~${(R_i \colon i \in J)}$ is a cycle by Lemma~\ref{l:raygraphsarecycles}, 
    and it is obtained from~$K$ by subdividing edges, which doesn't affect the cyclic order. 
    Hence it is sufficient to show that there is some cyclic orientation of~$K_J$ such that~$\sigma$ preserves the cyclic orientation of~$K_{X,Y}$.
    
    Since each linkage inducing~$\sigma$ gives rise to a family of tendrils~$(S'_y \colon y \in Y)$ where~$S'y$ shares a tail with~$R_{\sigma(y)}$, it follows that if~$\sigma(y)$ and~$\sigma(y)'$ are adjacent in~$K_J$ then~$y$ and~$y'$ are adjacent in~$K_{X,Y}$. 
    Since both~$K_J$ and~$K_{X,Y}$ are cycles, it follows that there is some cyclic orientation of~$K_J$ such that~$\sigma$ preserves the cyclic orientation of~$K_{X,Y}$.
    
    Suppose for a contradiction that there are two, and let~$\sigma$ and~$\sigma'$ be transition functions witnessing that both orientations of the ray graph are correct. 
    By Lemma~\ref{l:transitionpebble} we may assume without loss of generality that the images of~$\sigma$ and~$\sigma'$ are the same. 
    Call this common image~$I'$. 
    Since the ray graphs of~${(R_i \colon i \in I)}$ and~${(R_i \colon i \in I')}$ are both cycles, the former is obtained from the latter by subdivision of edges. 
    Since this does not affect the cyclic order, we may assume without loss of generality that~${I' = I}$. 
    By Lemma~\ref{l:weaklink} again, there is some transition function $\tau$ from the~$R_i$ to the~$S_y$. By Lemma~\ref{l:raynopaths}, both~${\tau \cdot \sigma}$ and~${\tau \cdot \sigma'}$ must preserve the cyclic order, which is the desired contradiction. 
\end{proof}

It therefore makes sense to refer to {\em the} correct orientation of a ray graph.

\begin{corollary}
    \label{l:correctorientationcycletransition}
    Any transition function between two families of at least~${N + 3}$ $\epsilon$-rays preserves the correct orientations of their ray graphs. 
\end{corollary}

\begin{proof}
    Suppose that~${\mathcal{R} = (R_i \colon i\in I)}$ and~${\mathcal{T}=(T_j \colon j \in J)}$ are families of at least~${N+3}$ rays and~$\sigma$ is a transition function from~$\mathcal{R}$ to~$\mathcal{T}$.
    
    Let us fix some transition function~$\tau$ from~${(S_y \colon y \in Y)}$ to~$\mathcal{R}$ and let~$\mathcal{P}$ be a linkage from~${(S_y \colon y \in Y)}$ to~$\mathcal{R}$ which induces~$\tau$. 
    For any finite~${X \subseteq V(\Gamma)}$ there is a linkage~$\mathcal{P}'$ from~$\mathcal{R}$ to~$\mathcal{T}$ which is after~${\bigcup \mathcal{P} \cup X}$ and which induces~$\sigma$. 
    Then, ${((S_y \colon y \in Y) \circ_{\mathcal{P}} \mathcal{R}) \circ_{\mathcal{P'}} \mathcal{T}}$ is a linkage from~${(S_y \colon y \in Y)}$ to~$\mathcal{T}$ which is after~$X$ and induces~${\sigma \cdot \tau}$. 
    It follows that~${\sigma \cdot \tau}$ is a transition function from~${(S_y \colon y \in Y)}$ to~$\mathcal{T}$.
    
    However, by the definition of correct orientation and Lemma~\ref{l:correctorientationcyclic}, $\tau$ and~${\sigma \cdot \tau}$ both preserve the cyclic orientation of~$K_{X,Y}$, and hence~$\sigma$ must preserve the correct orientation of the ray graphs of~$\mathcal{R}$ and~$\mathcal{T}$.
\end{proof}

\subsection{Half-grid-like ends}

In this subsection we suppose that~$\epsilon$ is thick but neither pebbly nor grid-like. 
We shall call such ends {\em half-grid-like}, since as we shall shortly see in this case the ray graphs and the transition functions between them behave similarly to those for the unique end of the half-grid. 
Note that, by definition Theorem~\ref{t:classification} holds.

We will need to carefully consider how the ray graphs are divided up by their cutvertices. 
In particular, for a graph~$K$ and vertices~$x$ and~$y$ of~$K$ we will denote by~${C^{xy}(K)}$ the union of all components of~${K - x}$ which do not contain~$y$, 
and we will denote by~$K^{xy}$ the graph~${K - C^{xy}(K) - C^{yx}(K)}$. 
We will refer to~$K^{xy}$ as the part of~$K$ {\em between}~$x$ and~$y$. 

As in the last subsection, let~${(X, Y)}$ be a polypod of order~${N+2}$ attaining its connection graph and let~${(S_y \colon y \in Y)}$ be a family of tendrils for~${(X,Y)}$ with ray graph~$K_{X,Y}$, which by assumption is not a cycle. 
By Corollary~\ref{c:pebblyendstructure} there is a bare path of length at least~$1$ in~$K_{X,Y}$ all of whose edges are bridges. 
Let~${y_1y_2}$ be any edge of that path. 
Without loss of generality we have~${C^{y_1y_2}(K_{X,Y}) \neq \emptyset}$. 

Throughout the remainder of this section we will always consider arbitrary families ${\mathcal{R}=(R_i \colon i \in I)}$ of disjoint~$\epsilon$-rays with~${|I| \geq N+3}$. 
We will write~$K$ to denote the ray graph of~$\Rcal$.

\begin{remark}
    \label{r:prescut}
    For any transition function $\sigma$ from the $S_y$ to the $R_i$ we have the inclusions $\sigma[V(C^{y_1y_2}(K_{X,Y}))] \subseteq V(C^{\sigma(y_1)\sigma(y_2)}(K))$ and $\sigma[V(C^{y_2y_1}(K_{X,Y}))] \subseteq V(C^{\sigma(y_2)\sigma(y_1)}(K))$ by Lemma~\ref{l:raysubgraph}. 
    Thus~$\sigma[Y]$ and~${V(K^{\sigma(y_1)\sigma(y_2)})}$ meet precisely in~$\sigma(y_1)$ and~$\sigma(y_2)$.
\end{remark}

\begin{lemma}
    \label{l:centralpathbare}
    For any transition function~$\sigma$ from the~$S_y$ to the~$R_i$ the graph~${K^{\sigma(y_1)\sigma(y_2)}}$ is a path from~$\sigma(y_1)$ to~$\sigma(y_2)$. 
    This path is a bare path in~$K$ and all of its edges are bridges.
\end{lemma}

\begin{proof}
    Since~$K$ is connected, ${K^{\sigma(y_1)\sigma(y_2)}}$ must include a path~$P$ from~$\sigma(y_1)$ to~$\sigma(y_2)$. 
    If it is not equal to that path then it follows from Lemma~\ref{l:transitionpebble} that the function~$\sigma'$, which we define to be just like~$\sigma$ except for~${\sigma'(y_1) = \sigma(y_2)}$ and~${\sigma'(y_2) = \sigma(y_1)}$, is a transition function from the~$S_y$ to the~$R_i$. 
    But then by Remark~\ref{r:prescut} we have 
    \begin{align*}
        \sigma[V(C^{y_1y_2}(K_{X,Y}))] 
        &\subseteq V(C^{\sigma(y_1)\sigma(y_2)}(K)) \cap V(C^{\sigma'(y_1)\sigma'(y_2)}(K)) \\
        &= V(C^{\sigma(y_1)\sigma(y_2)}(K)) \cap V(C^{\sigma(y_2)\sigma(y_1)}(K)) = \emptyset,
    \end{align*}
    a contradiction. 
    The last sentence of the lemma follows from the definition of~${K^{\sigma(y_1)\sigma(y_2)}}$.
\end{proof}

Given a path~$P$ with endvertices~$s$ and~$t$ we say {\em the orientation of~$P$ from~$s$ to~$t$} to mean the total order~$\leq$ on the vertices of~$P$ where~${a \leq b}$ if and only if~$a$ lies on~${sPb}$, in this case we say that~$a$ \emph{lies before}~$b$. 
Note that every path with at least one edge has precisely two orientations. 

Now, we fix a transition function~$\sigma_{\text{max}}$ from the~$S_y$ to the~$R_i$ so that the path \linebreak ${P := K^{\sigma_{\text{max}}(y_1)\sigma_{\text{max}}(y_2)}}$ is as long as possible. 
We call~$P$ the {\em central} path of~$K$ and the orientation of~$P$ from~${\sigma_{\text{max}}(y_1)}$ to~${\sigma_{\text{max}}(y_2)}$ the {\em correct} orientation.

We first note that, for large enough families of rays almost all of the ray graph lies on the central path.

\begin{lemma}
    \label{l:quantativecentral}
    At most~$N$ vertices of~$K$ are not on the central path.
\end{lemma}

\begin{proof}
    By Remark~\ref{r:prescut} we have ${\sigma_{\text{max}}[V(C^{y_1y_2}(K_{X,Y}))] \subseteq V\left(C^{\sigma_{\text{max}}(y_1)\sigma_{\text{max}}(y_2)}(K)\right)}$. 
    If it were a proper subset, then we would be able to use Lemma~\ref{l:transitionpebble} to produce a transition function in which this path is longer. 
    So we must have ${\sigma_{\text{max}}[V(C^{y_1y_2}(K_{X,Y}))] = V(C^{\sigma_{\text{max}}(y_1)\sigma_{\text{max}}(y_2)}(K))}$ and similarly ${\sigma_{\text{max}}[V(C^{y_2y_1}(K_{X,Y}))] = V(C^{\sigma_{\text{max}}(y_2)\sigma_{\text{max}}(y_1)}(K))}$. 
    However, since~${y_1y_2}$ is a bridge, 
    ${|V(C^{y_1y_2}(K_{X,Y})) \cup V(C^{y_2y_1}(K_{X,Y}))|=N}$ and so at most~$N$ vertices of~$K$ are not on the central path.
\end{proof}

We call~$P$ the {\em central} path of~$K$ and the orientation of~$P$ from~$\sigma_{\text{max}}(y_1)$ to~$\sigma_{\text{max}}(y_2)$ the {\em correct} orientation. 
We note the following simple corollary, which will be useful in later work.

\begin{corollary}
    \label{c:corerays}
    For any~${i \in I}$ if ${\RG(\mathcal{R}) - i}$ has precisely two components, each of size at least~${N+1}$, then~$i$ is an inner vertex of the central path of~$\RG(\mathcal{R})$. 
\end{corollary}

\begin{proof}
    By Lemma~\ref{l:quantativecentral} both components of~${\RG(\mathcal{R}) - i}$ contain a vertex of the central path. 
    However, since all the edges of the central path are bridges, it follows that~$i$ lies between these two vertices on the central path.
\end{proof}

We can in fact determine the central path and its correct orientation by considering the possible transition functions from the~$S_y$ to the~$R_i$.

\begin{lemma}
    \label{l:correctpaths}
    For any two vertices~$v_1$ and~$v_2$ of~$K$, there exists a transition function ${\sigma \colon V(K_{X,Y}) \to V(K)}$ with ${\sigma(y_1) = v_1}$ and ${\sigma(y_2) = v_2}$ if and only if~$v_1$ and~$v_2$ both lie on~$P$, with~$v_1$ before~$v_2$. 
\end{lemma}

\begin{proof}
    The `if' direction is clear by applying Lemma~\ref{l:transitionpebble} to~$\sigma_{\text{max}}$.
    For the `only if' direction, we begin by setting~${c_1 = |V(C^{y_1y_2}(K_{X,Y}))|}$ and~${c_2 = |V(C^{y_2y_1}(K_{X,Y}))|}$. 
    We enumerate ${V(C^{y_1y_2}(K_{X,Y}))}$ as~${y_3 \ldots y_{c_1 + 2}}$ and~${V(C^{y_2y_1}(K_{X,Y}))}$ as~${y_{c_1 + 3} \ldots y_{c_1 + c_2 + 2}}$. 
    Then for any ${(N+2)}$-tuple ${(x_1, \ldots, x_{N+2})}$ of distinct vertices which is achievable in the \linebreak ${(\sigma_{\text{max}}(y_1), \ldots, \sigma_{\text{max}}(y_{N+2}))}$-pebble-pushing game on~$K$ we must have the following three properties, since they are preserved by any single move:
    \begin{itemize}
        \item $x_1$ and~$x_2$ lie on~$P$, with~$x_1$ before~$x_2$.
        \item ${\{x_3, \ldots, x_{c_1 + 2}\} \subseteq V(C^{x_1x_2}(K))}$.
        \item ${\{x_{c_1 + 3}, \ldots, x_{c_1 + c_2 + 2}\} \subseteq V(C^{x_2x_1}(K))}$.
    \end{itemize}
    
    Now let~$\sigma$ be any transition function from the~$S_y$ to the~$R_i$. 
    Let ${(x_1, \ldots, x_{N+2})}$ be an ${(N+2)}$-tuple achievable in the ${(\sigma_{\text{max}}(y_1), \ldots, \sigma_{\text{max}}(y_{N+2}))}$-pebble-pushing game such that~${\{x_1, \ldots, x_{N+2}\} = \sigma[Y]}$. 
    By Lemma~\ref{l:transitionpebble}, the function~$\sigma'$ sending~$y_i$ to~$x_i$ for each ${i \leq N+2}$ is also a transition function and~${\sigma'[Y] = \sigma[Y]}$. 
    Let~$\tau$ be a transition function from~${(R_i \colon i \in \sigma[Y])}$ to the~$S_y$. 
    Then, by Lemma~\ref{l:barepath}, both~${\tau \cdot \sigma}$ and~${\tau \cdot \sigma'}$ keep both~$y_1$ and~$y_2$ fixed. 
    Thus~${\sigma(y_1) = \sigma'(y_1) = x_1}$ and~${\sigma(y_2) = \sigma'(y_2) = x_2}$. 
    As noted above, this means that~$\sigma(y_1)$ and~$\sigma(y_2)$ both lie on~$P$ with~$\sigma(y_1)$ before~$\sigma(y_2)$, as desired.
\end{proof}

Thus the central path and the correct orientation depend only on our choice of~$y_1$ and~$y_2$. 
Hence, we get the following corollary.
\begin{corollary}
    \label{cor:correct}
    Each ray graph on at least~${N+3}$ vertices contains a unique central path with a correct orientation 
    and every transition function between two families of at least~${N+3}$ $\epsilon$-rays sends vertices of the central path to vertices of the central path and preserves the correct orientation.
\end{corollary}

\begin{proof}
    Consider the family~${\mathcal{R}=(R_i\colon i \in I)}$ with its ray graph~$K$ and another family ${\mathcal{T}=(T_j \colon j \in J)}$ of at least~${N+3}$ rays, with ray graph~$K_\mathcal{T}$, 
    and let~$\tau$ be a transition function from~$\mathcal{R}$ to~$\mathcal{T}$.
    
    Let~${v_1, v_2}$ be two vertices in the central path~$P$ of~$K$ with~$v_1$ before~$v_2$. 
    By Lemma~\ref{l:correctpaths} there is transition function~$\sigma$ from~${(S_y \colon y \in Y)}$ to~$\mathcal{R}$ with~${\sigma(y_1) = v_1}$ and~${\sigma(y_2) = v_2}$. 
    
    Then, as in Lemma~\ref{l:correctorientationcycletransition}, it is clear that~${\tau \cdot \sigma}$ is a transition function from ${(S_y \colon y \in Y)}$ to~$\mathcal{T}$. 
    However since~${\tau \cdot \sigma (y_1) = \tau(v_1)}$ and~${\tau \cdot \sigma ( y_2) = \tau(v_2)}$, it follows from Lemma~\ref{l:correctpaths} that~$\tau(v_1)$ and~$\tau(v_2)$ both lie on the central path~$P_\mathcal{T}$ of~$K_\mathcal{T}$ with~$\tau(v_1)$ before~$\tau(v_2)$, and hence~$\tau$ sends vertices of~$P$ to vertices of~$P_\mathcal{T}$ and preserves the correct orientation.
\end{proof}

\begin{lemma}
    Let~$\mathcal{R}$ and~$\mathcal{T}$ be families of disjoint rays, each of size at least~${N+3}$, and let~$\sigma$ be a transition function from~$\Rcal$ to~$\Tcal$. 
    Let~${x \in \RG(\mathcal{R})}$ be an inner vertex of the central path. 
    If~${v_1,v_2\in \RG(\mathcal{R})}$ lie in different components of~${\RG(\mathcal{R})-x}$, then~${\sigma(v_1)}$ and~${\sigma(v_2)}$ lie in different components of~${\RG(\Tcal)-\sigma(x)}$. 
    Moreover, $\sigma(x)$ is an inner vertex of the central path of~${\RG(\Tcal)}$.
\end{lemma}

\begin{proof}
    That~$\sigma(x)$ is an inner vertex of the central path of~$\RG(\Tcal)$ follows from Corollary~\ref{cor:correct}. 
    We note, by Lemma~\ref{l:raysubgraph}, given any family of rays $\mathcal{K}$ and a transition function~$\gamma$ from~$\mathcal{S}$ to~$\mathcal{K}$, 
    if~$y$ separates~$x$ from~$z$ in~$K_{X,Y}$ then~$\gamma(y)$ separates~$\gamma(x)$ from~$\gamma(z)$ in~${\RG(\mathcal{K})}$.
    
    Let~${\tau \colon V(K_{X,Y})\to V(\RG(\Rcal))}$ be a transition function with~${\tau(y_1)=x}$ which exists by Lemma~\ref{l:correctpaths}. 
    Since~$x$ is an inner vertex of the central path of~${\RG(\Rcal)}$, there are exactly two components of~${\RG(\Rcal)-x}$, one containing~$v_1$ and one containing~$v_2$. 
    Furthermore, by Lemma~\ref{l:raysubgraph}, it follows that~${\tau(C^{y_1y_2}(K_{X,Y}))}$ and~${\tau(C^{y_2y_1}(K_{X,Y} \cup \{y_2\})}$ are contained in different components of~${\RG(\Rcal)-x}$.
    
    Hence, by Lemma~\ref{l:transitionpebble} we may assume without loss of generality that~${v_1,v_2\in \tau(V(K_{X,Y}))}$, where~$y_1$ separates~${w_1:=\tau^{-1}(v_1)}$ and~${w_2:=\tau^{-1}(v_2)}$ in~$K_{X,Y}$.
    
    However, by the remark above applied to the transition function~${\sigma \cdot \tau}$ we conclude that~${\sigma(x) = \sigma \cdot \tau(y_1)}$ separates~${\sigma(v_1) = \sigma \cdot \tau(w_1)}$ from~${\sigma(v_2) = \sigma \cdot \tau (w_2)}$.
\end{proof}

\section{%
\texorpdfstring{$G$-tribes and concentration of $G$-tribes towards an end}%
{G-tribes and concentration of G-tribes towards an end}%
}
\label{s:tribes}

To show that a given graph~$G$ is $\preceq$-ubiquitous, we shall assume that~${n G\preceq \Gamma}$ holds for every~${n \in \N}$ an show that this implies~${\aleph_0 G\preceq \Gamma}$. 
To this end we use the following notation for such collections of~$nG$ in~$\Gamma$, most of which we established in~\cite{BEEGHPTI}.

\begin{definition}[\Gtribe s]
    Let $G$ and $\Gamma$ be graphs.
    \begin{itemize}
        \item A \emph{\Gtribe\ in~$\Gamma$ (with respect to the minor relation)} is a family~$\mathcal{F}$ of finite collections~$F$ of disjoint subgraphs~$H$ of~$\Gamma$ such that each \emph{member}~$H$ of~$\Fcal$ is an~$IG$.
        \item A \Gtribe\ $\mathcal{F}$ in~$\Gamma$ is called \emph{thick}, if for each~${n \in \mathbb{N}}$ there is a \emph{layer} ${F \in \mathcal{F}}$ with~${|F| \geq n}$; otherwise, it is called \emph{thin}.
        \item A \Gtribe\ $\mathcal{F}'$ in~$\Gamma$ is a \emph{\Gsubtribe}\footnote{When~$G$ is clear from the context we will often refer to a $G$-subtribe as simply a subtribe.} 
            of a \Gtribe\ $\mathcal{F}$ in~$\Gamma$, denoted by~${\Fcal' \preceq \Fcal}$, if there is an injection~${\Psi\colon\Fcal'\to \Fcal}$ such that for each~${F' \in \mathcal{F}'}$ there is an injection~${\varphi_{F'} \colon F' \to \Psi(F')}$ such that~${V(H') \subseteq V(\varphi_{F'}(H'))}$
            for each~${H' \in F'}$. 
            The \Gsubtribe\ $\Fcal'$ is called \emph{flat}, denoted by~${\Fcal' \subseteq \Fcal}$, if there is such an injection~$\Psi$ satisfying~${F' \subseteq \Psi(F')}$.
        \item A thick $G$-tribe $\mathcal{F}$ in~$\Gamma$ is \emph{concentrated at an end~$\epsilon$} of~$\Gamma$, if for every finite set~$X$ of vertices of~$\Gamma$, the $G$-tribe ${\Fcal_X = \{F_X \colon F \in \Fcal \}}$ consisting of the layers \linebreak ${F_X=\{H \in F: H \not\subseteq C(X,\epsilon)\}\subseteq F}$ is a thin subtribe of~$\Fcal$.
            It is \emph{strongly concentrated at~$\epsilon$} if additionally, for every finite vertex set~$X$ of~$\Gamma$, every member~$H$ of~$\Fcal$ intersects~${C(X,\epsilon)}$. 
    \end{itemize}
\end{definition}

We note that every thick $G$-tribe~$\Fcal$ contains a thick subtribe~$\Fcal '$ such that every~${H \in \bigcup\Fcal}$ is a tidy~$IG$. 
We will use the following lemmas from~\cite{BEEGHPTI}.

\begin{lemma}[Removing a thin subtribe, {\cite[Lemma~5.2]{BEEGHPTI}}]
    \label{l:removethin}
    Let~$\Fcal$ be a thick $G$-tribe in~$\Gamma$ and let~$\Fcal'$ be a thin subtribe of~$\Fcal$, witnessed by~${\Psi\colon \Fcal'\to \Fcal}$ and~${(\varphi_{F'} \colon F' \in \mathcal{F}')}$. 
    For~${F \in \Fcal}$, if~${F \in \Psi(\Fcal')}$, let~${\Psi^{-1}(F)=\{F'_F\}}$ and set~${\hat{F}=\varphi_{F'_F}(F'_F)}$. 
    If~${F \notin \Psi(\Fcal')}$, set~${\hat{F}=\emptyset}$. Then 
    \[
        \Fcal'':=\{F\setminus \hat{F}\colon F\in \Fcal\}
    \] 
    is a thick flat $G$-subtribe of~$\Fcal$.
\end{lemma}

\begin{lemma}[Pigeon hole principle for thick $G$-tribes, {\cite[Lemma~5.3]{BEEGHPTI}}]
    \label{Lem_finitechoice}
    Suppose for some~${k \in \N}$, we have a $k$-colouring~${c\colon \bigcup \mathcal{F} \to [k]}$ of the members of some thick $G$-tribe~$\mathcal{F}$ in~$\Gamma$. 
    Then there is a monochromatic, thick, flat $G$-subtribe~$\mathcal{F}'$ of~$\mathcal F$.
\end{lemma}

Note that, in the following lemmas, it is necessary that~$G$ is connected, so that every member of the $G$-tribe is a connected graph.

\begin{lemma}[{\cite[Lemma~5.4]{BEEGHPTI}}]
    \label{l:concentrated}
    Let~$G$ be a connected graph and~$\Gamma$ a graph containing a thick $G$-tribe~$\mathcal{F}$. 
    Then either~${\aleph_0 G \preceq \Gamma}$, or there is a thick flat subtribe~$\Fcal'$ of~$\Fcal$ and an end~$\epsilon$ of~$\Gamma$ such that~$\Fcal'$ is concentrated at~$\epsilon$. 
\end{lemma}

\begin{lemma}[{\cite[Lemma~5.5]{BEEGHPTI}}]
    \label{lem_subtribesinheritconcentration}
    Let~$G$ be a connected graph and~$\Gamma$ a graph containing a thick $G$-tribe~$\mathcal{F}$ concentrated at an end~$\epsilon$ of~$\Gamma$. 
    Then the following assertions hold: 
    \begin{enumerate}
        \item For every finite set~$X$, the component~${C(X,\epsilon)}$ contains a thick flat $G$-subtribe of~$\Fcal$.
        \item Every thick subtribe~$\Fcal'$ of~$\Fcal$ is concentrated at~$\epsilon$, too.
    \end{enumerate}
\end{lemma} 

\begin{lemma}
    \label{l:strongconcentrated}
    Let~$G$ be a connected graph and~$\Gamma$ a graph containing a thick $G$-tribe~$\Fcal$ concentrated at an end~${\epsilon \in \Omega(\Gamma)}$. 
    Then either~${\aleph_0 G \preceq \Gamma}$, 
    or there is a thick flat subtribe of~$\Fcal$ which is strongly concentrated at~$\epsilon$. 
\end{lemma}

\begin{proof}
    Suppose that no thick flat subtribe of~$\Fcal$ is strongly concentrated at~$\epsilon$. 
    We construct an~${\aleph_0 G \preceq \Gamma}$ by recursively choosing disjoint $IG$s ${H_1, H_2, \ldots}$ in~$\Gamma$ as follows: 
    Having chosen~${H_1, H_2, \ldots, H_n}$ such that for some finite set~$X_n$ we have 
    \[
        H_i \cap C(X_n, \epsilon) = \emptyset
    \]
    for all~${i \in [n]}$, then by Lemma~\ref{lem_subtribesinheritconcentration}(1), there is still a thick flat subtribe~$\mathcal{F}'_n$ of~$\Fcal$ contained in~${C(X_n,\epsilon)}$. 
    Since by assumption, $\mathcal{F}'_n$ is not strongly concentrated at~$\epsilon$, we may pick~${H_{n+1}\in \mathcal{F}'_n}$ and a finite set~${X_{n+1} \supseteq X_n}$ with~${H_{n+1} \cap C(X_{n+1},\epsilon) =\emptyset}$. 
    Then the union of all the~$H_i$ is an~${\aleph_0 G \preceq \Gamma}$.
\end{proof}

The following lemma will show that we can restrict ourselves to thick $G$-tribes which are concentrated at thick ends.

\begin{lemma}
    \label{l:concentratedatthin}
    Let~$G$ be a connected graph and~$\Gamma$ a graph containing a thick $G$-tribe~$\Fcal$ concentrated at an end~${\epsilon \in \Omega(\Gamma)}$ which is thin. 
    Then~${\aleph_0 G \preceq \Gamma}$. 
\end{lemma}

\begin{proof}
    Since~$\epsilon$ is thin, we may assume by Proposition~\ref{p:infdomTKaleph} that only finitely many vertices dominate~$\epsilon$.
    Deleting these yields a subgraph of~$\Gamma$ in which there is still a thick $G$-tribe concentrated at~$\epsilon$.
    Hence we may assume without loss of generality that~$\epsilon$ is not dominated by any vertex in~$\Gamma$.
    
    Let~${k \in \mathbb{N}}$ be the degree of~$\epsilon$.
    By~\cite[Corollary~5.5]{GH18} there is a sequence of
    vertex sets~${(S_n \colon n \in \mathbb{N})}$ such that:
    \begin{itemize}
        \item ${|S_n|=k}$,
        \item ${C(S_{n+1},\epsilon) \subseteq C(S_n,\epsilon)}$, and 
        \item ${\bigcap_{n \in \mathbb{N}} C(S_n,\epsilon) = \emptyset}$.
    \end{itemize}
    
    Suppose there is a thick subtribe~$\Fcal'$ of~$\Fcal$ which is strongly concentrated at~$\epsilon$. 
    For any~${F \in \Fcal'}$ there is an~${N_F \in \mathbb{N}}$ such that~${H \setminus C(S_{N_F},\epsilon) \neq \emptyset}$ for all~${H \in F}$ by the properties of the sequence. 
    Furthermore, since~$\Fcal'$ is strongly concentrated, ${H \cap C(S_{N_F},\epsilon) \neq \emptyset}$ as well for each~${H \in F}$. 
    
    Let~${F \in \Fcal'}$ be such that~${|F| > k}$. 
    Since~$G$ is connected, so is~$H$, and so from the above it follows that~${H \cap S_{N_F} \neq \emptyset}$ for each~${H \in F}$, contradicting the fact that~${|S_{N_F}| = k < |F|}$. 
    Thus~${\aleph_0 G \preceq \Gamma}$ by Lemma~\ref{l:strongconcentrated}.
\end{proof}

Note that, whilst concentration is hereditary for subtribes, strong concentration is not. 
However if we restrict to \emph{flat} subtribes, then strong concentration is a hereditary property.

Let us show see how ends of the members of a strongly concentrated tribe relate to ends of the host graph~$\Gamma$. 
Let~$G$ be a connected graph and~${H \subseteq \Gamma}$ an~$IG$. 
By Lemmas~\ref{l:connraygraph} and~\ref{l:rayinducedsubgraph}, if~${\omega \in \Omega(G)}$ and~$R_1$ and~${R_2 \in \omega}$ then the pullbacks~${H^{\downarrow}(R_1)}$ and~${H^{\downarrow}(R_2)}$ belong to the same end~${\omega' \in \Omega(\Gamma)}$. 
Hence, $H$ determines for every end~${\omega \in G}$ a \emph{pullback end}~${H(\omega) \in \Omega(\Gamma)}$.
The next lemma is where we need to use the assumption that~$G$ is locally finite.

\begin{lemma}
    \label{l:locfincon}
    Let~$G$ be a locally finite connected graph and~$\Gamma$ a graph containing a thick $G$-tribe~$\Fcal$ strongly concentrated at an end~${\epsilon \in \Omega(\Gamma)}$, where every member is a tidy~$IG$. 
    Then either~${\aleph_0 G \preceq \Gamma}$, or there is a flat subtribe~$\Fcal'$ of~$\Fcal$ such that for every~${H \in \bigcup \Fcal'}$ there is an end~${\omega_H \in \Omega(G)}$ such that~${H(\omega_H) = \epsilon}$. 
\end{lemma}

\begin{proof}
    Since~$G$ is locally finite and every~${H \in \bigcup\Fcal}$ is tidy, the branch sets~${H(v)}$ are finite for each~${v \in V(G)}$. 
    If~$\epsilon$ is dominated by infinitely many vertices, then~${\aleph_0 G \preceq \Gamma}$ by Proposition~\ref{p:infdomTKaleph}, since every locally finite connected graph is countable. 
    If this is not the case, then there is some~${k \in \mathbb{N}}$ such that~$\epsilon$ is dominated by~$k$ vertices and so for every~${F \in \Fcal}$ at most~$k$ of the~${H \in F}$ contain vertices which dominate~$\epsilon$ in~$\Gamma$. 
    Therefore, there is a thick flat subtribe~$\Fcal'$ of~$\Fcal$ such that no~${H \in \bigcup\Fcal'}$ contains a vertex dominating~$\epsilon$ in~$\Gamma$. 
    Note that~$\Fcal'$ is still strongly concentrated at~$\epsilon$, and every branch set of every~${H \in \bigcup\Fcal'}$ is finite. 
    
    Since~$\Fcal'$ is strongly concentrated at~$\epsilon$, for every finite vertex set~$X$ of~$\Gamma$, every~${H \in \bigcup\Fcal'}$ intersects~${C(X,\epsilon)}$. 
    By a standard argument, since~$H$ as a connected infinite graph does not contain a vertex dominating~$\epsilon$ in~$\Gamma$, instead~$H$ contains a ray~${R_H \in \epsilon}$. 
    
    Since each branch set~${H(v)}$ is finite, $R_H$ meets infinitely many branch sets. 
    Let us consider the subgraph~${K \subseteq G}$ consisting of all the edges~${(v,w)}$ such that~$R_H$ uses an edge between~$H(v)$ and~$H(w)$. 
    Note that, since there is a edge in~$H$ between~$H(v)$ and~$H(w)$ if and only if~${(v,w) \in E(G)}$, $K$ is well-defined and connected.
    
    $K$ is then an infinite connected subgraph of a locally finite graph, and as such contains a ray~$S_H$ in~$G$. Since the edges between~$H(v)$ and~$H(w)$, if they exist, were unique, it follows that the pullback~$H^{\downarrow}(S_H)$ of~$S_H$ has infinitely many edges in common with~$R_H$, and so tends to~$\epsilon$ in~$\Gamma$. 
    Therefore, if~$S_H$ tends to~$\omega_H$ in~$\Omega(G)$, then~${H(\omega_H) = \epsilon}$. 
\end{proof}

\section{Ubiquity of minors of the half-grid}\label{s:halfgrid}
Here, and in the following, we denote by $\Half$ the infinite, one-ended, cubic hexagonal half-grid (see Figure~\ref{f:hex}). The following theorem of Halin is one of the cornerstones of infinite graph theory.

\begin{figure}[ht]
    \center
    \includegraphics[scale=1.2]{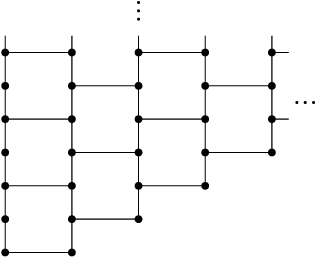}
    \caption{The hexagonal half-grid $\Half$.}
    \label{f:hex}
\end{figure}

\begin{theorem}[Halin, see {\cite[Theorem~8.2.6]{D16}}]
    \label{thm_Halin}
    Whenever a graph~$\Gamma$ contains a thick end, then~${\Half \leq \Gamma}$. \qed
\end{theorem}

In~\cite{halin1975problem}, Halin used this result to show that every topological minor of~$\Half$ is ubiquitous with respect to the topological minor relation~$\leq$. 
In particular, trees of maximum degree~$3$ are ubiquitous with respect to~$\leq$.

However, the following argument, which is a slight adaptation of Halin's, shows that every connected minor of~$\Half$ is ubiquitous with respect to the minor relation. 
In particular, the dominated ray, the dominated double ray, and all countable trees are ubiquitous with respect to the minor relation.

The main difference to Halin's original proof is that, since he was only considering locally finite graphs, he was able to assume that the host graph~$\Gamma$ was also locally finite. 

We will need the following result of Halin.

\begin{lemma}[{\cite[(4) in Section 3]{halin1975problem}}]
    \label{lem_halfgridubiquity}
    ${\aleph_0 \Half}$ is a topological minor of~$\Half$.
\end{lemma}

\halfgrid*
\begin{proof}
    Suppose~${G \preceq \mathbb{N} \square \mathbb{Z}}$ is a minor of the half-grid, and~$\Gamma$ is a graph such that~${n G \preceq \Gamma}$ for each~${n \in \mathbb{N}}$. 
    By Lemma~\ref{l:concentrated} we may assume there is an end~$\epsilon$ of~$\Gamma$ and a thick $G$-tribe~$\Fcal$ which is concentrated at~$\epsilon$. 
    By Lemma~\ref{l:concentratedatthin} we may assume that~$\epsilon$ is thick.
    Hence~${\Half \leq \Gamma}$ by Theorem~\ref{thm_Halin}, and with Lemma~\ref{lem_halfgridubiquity} we obtain
    \[
        \aleph_0 G \preceq \aleph_0 (\mathbb{N} \square \mathbb{Z}) \preceq \aleph_0 \Half \leq \Half \leq \Gamma. \qedhere 
    \]
\end{proof}

\begin{lemma}
    \label{lemma_treeembeddinghalfgrid} 
    $\Half$ contains every countable tree as a minor.
\end{lemma}

\begin{proof}
    It is easy to see that the infinite binary tree~$T_2$ embeds into~$\Half$ as a topological minor. 
    It is also easy to see that countably regular tree~$T_\infty$ where every vertex has infinite degree embeds into~$T_2$ as a minor. 
    And obviously, every countable tree~$T$ is a subgraph of~$T_\infty$. 
    Hence we have
    \[ 
        T \subseteq T_\infty \preceq T_2 \leq \Half  
    \]
    from which the result follows.
\end{proof}

\begin{corollary}
    \label{cor_trees}
    All countable trees are ubiquitous with respect to the minor relation.
\end{corollary}

\begin{proof}
    This is an immediate consequence of Lemma~\ref{lemma_treeembeddinghalfgrid} and Theorem~\ref{t:halfgrid}.
\end{proof}

\section{Proof of main results}\label{s:nonlinear}

The following technical result contains most of the work for the proof of Theorem~\ref{t:nonlin}, but is stated so as to be applicable in a later paper~\cite{BEEGHPTIII}.

\begin{lemma}
    \label{l:linearsubtribe}
    Let~$\epsilon$ be a non-pebbly end of~$\Gamma$ and let~$\Fcal$ be a thick $G$-tribe such that for every~${H \in \bigcup \Fcal}$ there is an end~${\omega_H \in \Omega(G)}$ such that~${H(\omega_H) = \epsilon}$. 
    Then there is a thick flat subtribe~$\Fcal'$ of~$\Fcal$ such that~$\omega_H$ is linear for every~${H \in \bigcup \Fcal'}$. 
\end{lemma}

\begin{proof}
    Let~$\Fcal''$ be the flat subtribe of~$\Fcal$ given by~${\Fcal'' = \{F'' \colon F \in \Fcal\}}$ with
    \[
        F'' = \{ H \colon H \in F \text{ and } \omega_H \text{ is not linear} \}.
    \]
    Suppose for a contradiction that~$\Fcal''$ is thick. 
    Then, there is some~${F \in \Fcal}$ which contains~${k+2}$ disjoint~$IG$s, ${H_1, H_2, \ldots, H_{k+2}}$, where~$k$ is such that~$\epsilon$ is not $k$-pebble-win. 
    By assumption~$\omega_{H_i}$ is not linear for each~$i$, and so for each~$i$ there is a family of disjoint rays ${\{R^i_1, R^i_2, \ldots, R^i_{m_i} \}}$ in~$G$ tending to~$\omega_{H_i}$ whose ray graph in~$G$ is not a path. 
    Let 
    \[
        \Scal = (H_i^{\downarrow}(R^i_j) \colon i \in [k+2], j \in [m_i]).
    \]
    By construction,~$\Scal$ is a disjoint family of $\epsilon$-rays in~$\Gamma$, and by Lemmas~\ref{l:pullbackraygraph} and~\ref{l:rayinducedsubgraph}, ${\RG_{\Gamma}(\Scal)}$ contains disjoint subgraphs ${K_1,K_2, \ldots, K_{k+2}}$ such that ${K_i \cong \RG_G(R^i_j \colon j \in [m_i])}$. 
    However, by Corollary~\ref{c:pebblyendstructure}, there is a set~$X$ of vertices of size at most~$k$ such that~${\RG_{\Gamma}(\Scal) - X}$ is a bare path~$P$. 
    However, then some~${K_i \subseteq P}$ is a path, a contradiction.   
    
    Since~$\Fcal$ is the union of~$\Fcal''$ and~$\Fcal'$ where~${\Fcal' = \{F' \colon F \in \Fcal\}}$ with
    \[
        F' = \{ H \colon H \in F \text{ and } \omega_H \text{ is linear} \},
    \] 
    it follows that~$\Fcal'$ is thick.
\end{proof}

\nonlin*
\begin{proof}
    Let~$\Gamma$ be a graph such that~${nG \preceq \Gamma}$ holds for every~${n \in \mathbb{N}}$. 
    Hence, $\Gamma$ contains a thick $G$-tribe~$\Fcal$. 
    By Lemmas~\ref{l:concentrated} and~\ref{l:strongconcentrated} we may assume that~$\Fcal$ is strongly concentrated at an end~$\epsilon$ of~$\Gamma$ and so by Lemma~\ref{l:locfincon} we may assume that for every~${H \in \bigcup\Fcal}$ there is an end~${\omega_H \in \Omega(G)}$ such that~${H(\omega_H) = \epsilon}$. 
    
    Since~$\omega_H$ is not linear for each~${H \in \bigcup \Fcal}$, it follows by Lemma~\ref{l:linearsubtribe} that~$\epsilon$ is pebbly, 
    and hence by Corollary~\ref{c:pebblyubiq} ${\aleph_0 G \preceq \Gamma}$.
\end{proof}

\begin{figure}[ht]
    \center
    \begin{tikzpicture}[scale=.4]
\draw[help lines] (-11.25,-7.25) grid (11.25,11.25);

\tikzstyle{ray}=[->,thick]
\tikzstyle{connection}=[red,thick,dash pattern=on 5pt off 2pt]

\draw[ray] (0,1) -- (0, 11.75);
\draw[ray] (1,1) |- (3,5) -- (3,11.75);
\draw[ray] (1,0) -- (4, 0) -- (4,7) -- (11.75,7);
\draw[ray] (0,0) |- (11.75,-1);
\draw[ray] (-2,0) -- (-2, -3) -- (5,-3) -- (5,-7.75);
\draw[ray] (-4,2) -- (-6, 2) -- (-6,-2) -| ++(-1,-1) -| ++(-1,-1) -| ++(-1,-1) -| ++(-1,-1) -| ++(-1,-1) -- ++(-.75,0);
\draw[ray] (-2,3) -- (-11.75, 3);

\draw[connection] (-5, 3.1) .. controls (-5,7) .. (-.1, 7);
\draw[connection] (-8, 3.1) .. controls (-8,9) .. (-.1, 9);
\draw[connection] (-10, 3.1) .. controls (-10,11) .. (-.1, 11);

\draw[connection] (.1,6) -- (2.9,6);
\draw[connection] (.1,8) -- (2.9,8);
\draw[connection] (.1,10) -- (2.9,10);

\draw[connection] (3.1,7) -- (3.9,7);
\draw[connection] (3.1,9) .. controls (7,9) .. (7,7.1);
\draw[connection] (3.1,11) .. controls (10,11) .. (10,7.1);

\draw[connection] (4.25,2) .. controls (6,2) ..  (6, -.9);
\draw[connection] (8,6.9) -- (8, -.9);
\draw[connection] (11,6.9) -- (11, -.9);

\draw[connection] (5, -1.1) -- (5, -2.9);
\draw[connection] (7, -1.1) .. controls (7,-5) .. (5.1, -5);
\draw[connection] (10, -1.1) .. controls (10,-7) .. (5.1, -7);

\draw[connection] (-2.1,-2) -- (-5.9, -2);
\draw[connection] (4.9,-4) -- (-7.9, -4);
\draw[connection] (4.9,-6) -- (-9.9, -6);

\draw[connection] (-6.1, -1) .. controls (-7,-1) .. (-7, 2.9);
\draw[connection] (-9, -3.9) -- (-9, 2.9);
\draw[connection] (-11, -5.9) -- (-11, 2.9);
    \end{tikzpicture}
    \caption{The ray graphs in the full-grid are cycles.}
    \label{fig:grid}
\end{figure}

\fullgrid*
\begin{proof}
    Let~$G$ be the full-grid. 
    Note that~$G$ has a unique end and, furthermore, ${G - R}$ has at most one end for any ray~${R \in G}$. 
    It follows by Lemma~\ref{l:connraygraph} that the ray graph of any finite family of three or more rays is $2$-connected. 
    Hence, the unique end of~$G$ is non-linear and so, by Theorem~\ref{t:nonlin}, $G$ is $\preceq$-ubiquitous
\end{proof}

\begin{remark}
    In fact, every ray graph in the full-grid is a cycle (see Figure~\ref{fig:grid}).
\end{remark}

\gridprod*
\begin{proof}
    If~$G$ is a path or a ray, then~${G \square \mathbb{Z}}$ is a subgraph of the half-grid~${\N\square\Z}$ and thus $\preceq$\nobreakdash-ubiquitous by Theorem~\ref{t:halfgrid}. 
    If~$G$ is a double ray, then~${G \square \mathbb{Z}}$ is the full-grid and thus $\preceq$\nobreakdash-ubiquitous by Corollary~\ref{c:fullgrid}. 
    
    Otherwise, let~$G'$ be a finite connected subgraph of~$G$ which is not a path and let~$H$ be~$\mathbb{Z}$ or~$\mathbb{N}$. 
    We note first that~${G \square H}$ has a unique end. Furthermore, for any ray~$R$ of~$H$ it is clear that~$G'$ is a subgraph of~${\RG_{G \square H}((\{v\} \square R)_{v \in V(G')})}$, and so this ray graph is not a path. 
    Hence by Lemma~\ref{l:linearraygraph}, ${G \square H}$ has nowhere-linear end structure and is therefore $\preceq$-ubiquitous by Theorem~\ref{t:nonlin}. 
\end{proof}

Finally let us prove Theorem~\ref{t:dominatedray}. 
Recall that for~${k \in \mathbb{N}}$ we let~$DR_k$ denote the graph formed by taking a ray~$R$ together with~$k$ vertices~${v_1,v_2, \ldots, v_k}$ adjacent to every vertex in~$R$. 
We shall need the following strengthening of Proposition~\ref{p:rayorinfvertex}. 

A \emph{comb} is a union of a ray~$R$ with infinitely many disjoint finite paths, all having precisely their first vertex on~$R$. 
The last vertices of these paths are the \emph{teeth} of the comb.
 
\begin{proposition}
    \label{p:starcomb}\cite[Proposition~8.2.2]{D16}
    Let~$U$ be an infinite set of vertices in a connected graph~$G$. 
    Then~$G$ either contains a comb with all teeth in~$U$ or a subdivision of an infinite star with all leaves in~$U$.
\end{proposition}

\dominatedray*
\begin{proof}
    Let~${R = x_1 x_2 x_3\ldots}$ be the ray as stated in the definition of~$DR_k$ and let ${v_1,v_2, \ldots, v_k}$ denote the vertices adjacent to each vertex of~$R$. 
    Note that if~${k \leq 2}$ then~$DR_k$ is a minor of the half-grid, and hence $\preceq$-ubiquity follows from Theorem~\ref{t:halfgrid}. 
    
    Suppose then that~${k \geq 3}$ and~$\Gamma$ is a graph which contains a thick $DR_k$-tribe~$\Fcal$ each of whose members is tidy. 
    We may further assume, without loss of generality, that for each~${H \in \bigcup \Fcal}$, each~${i \in [k]}$, and each vertex~$x$ of~$H(v_i)$, every component of~${H(v_i)-x}$ contains a vertex~$y$ such that there is some vertex~${r \in R}$ and vertex~${z \in H(r)}$ with~$yz$ the unique edge between~$H(v_i)$ and~$H(r)$
    
    By Lemma~\ref{l:strongconcentrated} we may assume that there is an end~$\epsilon$ of~$\Gamma$ such that~$\Fcal$ is strongly concentrated at~$\epsilon$. 
    If there are infinitely many vertices dominating~$\epsilon$, then~${\aleph_0 DR_k \preceq K_{\aleph_0} \leq \Gamma}$ holds by Proposition~\ref{p:infdomTKaleph}. 
    So, we may assume that only finitely many vertices dominate~$\epsilon$. 
    By taking a thick subtribe if necessary, we may assume that no member of~$\Fcal$ contains such a vertex. 
    
    As before, if we can show that~$\epsilon$ is pebbly, then we will be done by Corollary~\ref{c:pebblyubiq}. 
    So suppose for a contradiction that~$\epsilon$ is not $r$-pebble-win for some~${r \in \mathbb{N}}$.
    
    We first claim that for each~${H \in \mathcal{F}}$ the pullback~${R_H = H^{\downarrow}(R)}$ of~$R$ in~$H$ is an~$\epsilon$-ray. 
    Indeed, since~$\mathcal{F}$ is strongly concentrated at~$\epsilon$ for every finite vertex set~$X$ of~$\Gamma$, $H$ intersects~${C(X,\epsilon)}$. 
    As in Lemma~\ref{l:locfincon}, since~$H$ is a connected graph and does not contain a vertex dominating~$\epsilon$ in~$\Gamma$, $H$ must contain a ray~${S \in \epsilon}$. 
    If~$S$ meets infinitely many branch sets then it must meet infinitely many branch sets of the form~${H(x_i)}$ for some~$x$ and hence, since~$R_H$ meets every~${H(x_i)}$, which are all connected subgraphs, we have that~${R_H \sim S}$ and so~${R_H \in \epsilon}$. Conversely, if~$S$ meets only finitely many branch sets then there must be some~$v_i$ such that~${H(v_i)}$ contains a tail of~$S$. 
    By our assumption on~$H(v_i)$, for any tail of~$S$ the component of~$H(v_i)$ containing that tail meets some edge between~$H(v_i)$ and some~$H(x_j)$. 
    In this case it is also easy to see that~${S \sim R_H}$, and so~${R_H \in \epsilon}$. 
    
    For each~${H \in \bigcup \Fcal}$ and each~${i \in [k]}$ we have that~$H(v_i)$ is a connected subgraph of~$\Gamma$. 
    Let~$U$ be the set of all vertices in~$H(v_i)$ which are the endpoint of some edge in~$H$ between~$H(v_i)$ and~$H(w)$ with~${w \in R}$. 
    Since~$v_i$ dominates~$R$, $U$ is infinite, and so by Proposition~\ref{p:starcomb}, $H(v_i)$ either contains a comb with all teeth in~$U$ or a subdivision of an infinite star with all leaves in~$U$. 
    However in the latter case the centre of the star would dominate~$\epsilon$, and so each~$H(v_i)$ contains such a comb, whose spine we denote by~$R_{H,i}$. 
    Now we set~${\mathcal{R}_H = (R_{H,1}, R_{H,2}, \ldots, R_{H,k}, R_H)}$.
    
    Since~$R_{H,i}$ is the spine of a comb, all of whose leaves are in~$U$, it follows that in the graph~${\RG_{H}(\mathcal{R}_H)}$ each~$R_{H,i}$ is adjacent to~$R_H$. 
    Hence~${\RG_{H}(\mathcal{R}_H)}$ contains a vertex of degree~${k \geq 3}$.
    
    There is some layer~${F \in \mathcal{F}}$ of size~${\ell \geq r+1}$, say~${F = (H_i \colon i \in [\ell])}$. 
    For every~${i \in [r+1]}$ we set~${\mathcal{R}_{H_i} = (R_{H_i,1}, R_{H_i,2}, \ldots, R_{H_i,k}, R_{H_i})}$. 
    Let us now consider the family of disjoint rays
    \[
        \mathcal{R} = \bigcup_{i=1}^{r+1} \mathcal{R}_{H_i}. 
    \]
    
    By construction~$\Rcal$ is a family of disjoint rays which tend to~$\epsilon$ in~$\Gamma$ and by Lemmas~\ref{l:pullbackraygraph} and~\ref{l:rayinducedsubgraph}, $\RG_{\Gamma}(\Rcal)$ contains~${r+1}$ vertices whose degree is at least~${k \geq 3}$. 
    However, by Corollary~\ref{c:pebblyendstructure}, there is a vertex set~$X$ of size at most~$r$ such that~${\RG_{\Gamma}(\Rcal) - X}$ is a bare path~$P$. 
    But then some vertex whose degree is at least~$3$ is contained in the bare path, a contradiction.
\end{proof}

\bibliographystyle{abbrv}
\bibliography{main}

\end{document}